\documentclass[12pt]{article}

\usepackage{latexsym,amsfonts,amsmath,amssymb,pstricks,wasysym,stmaryrd,enumerate,epsfig}
\usepackage{mathrsfs}
\usepackage{theorem}
\usepackage{cite}
\usepackage[all,cmtip]{xy}

\numberwithin{equation}{section}
\newtheorem{definition}{Definition}[section]
\newtheorem{defnprop}{Definition-Proposition}[section]

\newtheorem{proposition}[definition]{Proposition}
\newtheorem{theorem}[definition]{Theorem}
\newtheorem{corollary}[definition]{Corollary}
\newtheorem{lemma}[definition]{Lemma}

\newcommand{\be}{\begin{equation}}
\newcommand{\ee}{\end{equation}}
\newcommand{\beu}{\begin{equation*}}
\newcommand{\eeu}{\end{equation*}}
\newcommand{\bea}{\begin{eqnarray}}
\newcommand{\eea}{\end{eqnarray}}
\newcommand{\beaa}{\begin{eqnarray*}}
\newcommand{\eeaa}{\end{eqnarray*}}
\newcommand{\bmx}{\begin{pmatrix}}
\newcommand{\emx}{\end{pmatrix}}

\newcommand{\g}{{\mathfrak g}}
\newcommand{\h}{{\mathfrak h}}

\newcommand{\p}{{\mathfrak p}}

\newcommand{\finproof}{{\hfill \rule{5pt}{5pt}}}

\newcommand{\mf}{\mathfrak}

\newcommand{\nn}{\nonumber}

\newcommand{\eps}{\epsilon}
\newcommand{\tr}{\,{\rm tr}}

\newcommand{\ip}{{\cdot}}

\newcommand{\F}{{\mathscr F}}
\newcommand{\R}{{\mathcal R}}
\newcommand{\id}{{\mathrm{id}}}

\newcommand{\U}{\mathcal U}
\newcommand{\Deltat}{\tilde{\Delta}}
\newcommand{\isom}{\stackrel{\sim}{\longrightarrow}}
\newcommand{\K}{\mathbb K}

\advance\textheight by \topskip
\begin{document}

\baselineskip 17.5pt
\parindent 18pt
\parskip 8pt

\begin{flushright}
\break

DCPT-08/73

\end{flushright}
\vspace{1cm}
\begin{center}
{\LARGE {\bf Triangular quasi-Hopf algebra structures}} 
\\[4mm]
{\LARGE {\bf on certain non-semisimple quantum groups}}

\vspace{1.5cm}
{\large  C. A. S. Young\footnote{\texttt{charlesyoung@cantab.net}} and R. Zegers\footnote{\texttt{robin.zegers@durham.ac.uk}}}
\\
\vspace{10mm}

{ \emph{Department of Mathematical Sciences\\ University of Durham\\
South Road, Durham DH1 3LE, UK}}

\end{center}

\vskip 1in
\centerline{\small\bf ABSTRACT}
\centerline{
\parbox[t]{5in}{\small One way to obtain Quantized Universal Enveloping Algebras (QUEAs) of non-semisimple Lie algebras is by contracting QUEAs of semisimple Lie algebras. We prove that every contracted QUEA in a certain class is a cochain twist of the corresponding undeformed universal envelope. Consequently, these contracted QUEAs possess a triangular quasi-Hopf algebra structure. As examples, we consider $\kappa$-Poincar\'e in 3 and 4 spacetime dimensions. }}


\vspace{1cm}

\newpage
\section{Introduction}
The class of quasi-triangular quasi-Hopf (qtqH) algebras, introduced by Drinfel'd \cite{Drinfeld}, {\nobreak admits} an extended notion of twisting in which the 2-cocycle condition required in the context of Hopf algebras is relaxed, \cite{Chari, Majidbook}. It is therefore possible to relate, by twisting, Hopf algebras -- with coassociative coproduct -- to quasi-Hopf algebras which may only be coassociative up to conjugation by an invertible element $\Phi$ known as the {\it coassociator}. A remarkable fact, proved in \cite{Drinfeld}, is that every qtqH quantized universal enveloping algebra (QUEA) is isomorphic to a twist of the undeformed UEA of the underlying Lie algebra; the latter being endowed with a canonical qtqH algebra structure $(\R_{\mbox{\scriptsize KZ}}, \Phi_{\mbox{\scriptsize KZ}})$, obtained from the monodromy of the Knizhnik-Zamolodchikov equation. In particular, the Drinfel'd-Jimbo QUEA of any semisimple Lie algebra, endowed with its standard quasi-triangular Hopf algebra structure, can be obtained in this way by means of an appropriate twist $\F_{\mbox{\scriptsize D}}$. As a corollary, every such QUEA also admits a triangular quasi-Hopf algebra structure, obtained by twisting $(\R_0 = 1 \otimes 1, \Phi_0 = 1 \otimes 1 \otimes 1)$ with the same canonical twist $\F_{\mbox{\scriptsize D}}$. This structure provides the isomorphisms required to make the category of representations a tensor category. As emphasized in \cite{YZ} -- see also \cite{YZ2,YZ3} -- apart from the obvious mathematical interest, the extension of this result to non-semisimple Lie algebras would provide a covariant notion of multiparticle states in quantum field theories based on certain deformations of the Poincar\'e symmetry. Unfortunately, the proof of this result relies crucially on the vanishing of a certain cohomology module, which holds for semisimple Lie algebras but may fail for non-semisimple Lie algebras. This precludes any systematic extension to the latter and, to our knowledge, the question of the existence of a qtqH algebra structure (triangular or not) on a general non-semisimple QUEA remains open.

There is a class of non-semisimple Lie algebras that is nonetheless closely related to semisimple Lie algebras. It consists of all the Lie algebras obtained by {\it contracting} semisimple Lie algebras. As we shall discuss, given a symmetric decomposition $\g=\h \oplus\p$ of the semisimple Lie algebra $\g$, an {\it In\"onu-Wigner contraction} of $\g$ can be performed by rescaling the submodule $\p$ with respect to the subalgebra $\h$ and taking a singular limit. In the limit, the submodule $\p$ is contracted to an abelian ideal of the contraction $\g_0$ of $\g$, thus making $\g_0$ non-semisimple. Whenever the contraction procedure is non-singular at the level of the QUEA, this yields a complete QUEA structure $(\U_\kappa (\g_0), \Delta_\kappa)$ based on the non-semisimple Lie algebra $\g_{0}$, where $\kappa$ is a rescaled deformation parameter. In this paper, we consider a certain class of QUEAs obtained in this way. We prove that every such QUEA is isomorphic to a twist of the corresponding undeformed UEA by an invertible element $\F_0 \in \U_\kappa (\g_{0} ) \otimes \U_\kappa (\g_{0})$ obtained as the contraction of the canonical twist $\F_{\mbox{\scriptsize D}}$ of $\g$. By twisting the trivial triangular quasi-Hopf algebra structure with $\F_0$, we prove the existence of a triangular quasi-Hopf algebra structure $(\R_\kappa, \Phi_\kappa)$ on the non-semisimple QUEA $(\U_\kappa \g_{0}, \Delta_\kappa)$. That is, we prove the existence of the bottom line of the following diagram.
$$\xymatrixcolsep{5pc} \xymatrix{
\left( \U\g, \Delta_0, \R_{\mbox{\scriptsize KZ}}, \Phi_{\mbox{\scriptsize KZ}} \right ) \ar[r]^{\F_{\mbox{\scriptsize D}}} & \left( \U_q\g, \Delta_q, \R_q, 1^{\otimes 3} \right )\\
\left( \U\g, \Delta_0, 1^{\otimes 2}, 1^{\otimes 3} \right ) \ar[d]_{\mbox{Contraction}} \ar[r]^{\F_{\mbox{\scriptsize D}}} & \left( \U_q \g, \Delta_q, \R, \Phi \right )\ar[d]\\
\left( \U(\g_0), \Delta_0, 1^{\otimes 2}, 1^{\otimes 3} \right ) \ar[r]^{\F_0} & \left( \U_\kappa (\g_0), \Delta_\kappa, \R_\kappa, \Phi_\kappa \right )}$$

The paper is organised as follows. In section \ref{SecI}, we recall the definition of {\it symmetric} semisimple Lie algebras. The important notion of contractibility is introduced in section \ref{SecII} after a brief reminder of the definitions of the filtered and graded algebras associated to UEAs. We also define the class of symmetric spaces to which our results will apply, namely those possessing what we shall call the restriction property. Section \ref{SecIII} is dedicated to the cohomology of associative algebras and Lie algebras. After a brief account of Hochschild and Chevalley-Eilenberg cohomology, we introduce the notion of contractible Chevalley-Eilenberg cohomology. We establish, in particular, the vanishing of the first contractible Chevalley-Eilenberg cohomology module for symmetric semisimple Lie algebras possessing the restriction property. This will be crucial in proving the existence of a contractible twist. In section \ref{SecIV}, the usual rigidity theorems for semisimple Lie algebras are then refined, with special regards to the contractibility of the structures. We construct, in particular, a contractible twist from every contractible QUEA of restrictive type to the undeformed UEA of the underlying Lie algebra. The actual contraction is performed in section \ref{SecV}. Section \ref{SecVI} contains the examples that form the main motivation for the present work, namely the $\kappa$-deformations of $\U(\mathfrak{iso}(3,\mathbb C))$ and $\U(\mathfrak{iso}(4,\mathbb C))$, whose real forms give rise to the $\kappa$-deformations of the Euclidean and Poincar\'e algebras in three and four dimensions, \cite{Celeghini, Lukierski}. The latter has indeed received considerable interest as a possible deformation of the Poincar\'e symmetries of space-time -- see {\it e.g.} \cite{Lots} and references therein.

Throughout this paper $\K$ denotes a field of characteristic zero.

\section{Symmetric decompositions of Lie algebras}
\label{SecI}
Let us briefly review some well-known facts concerning symmetric semisimple Lie algebras. Following \cite{Dixmier,Helgasonbook}, we have 
\begin{definition}
A symmetric Lie algebra is a pair $(\g, \theta)$, where $\g$ is a Lie algebra and $\theta:\g\rightarrow \g$ is an involutive (i.e. $\theta\circ\theta = \id$ and $\theta\neq \id$) automorphism of Lie algebras.
\end{definition}
As $\theta \circ \theta = \id$, the eigenvalues of $\theta$ are $+1$ and $-1$. Let $\h = \ker \left (\theta-\id \right )$ and $\p = \ker \left ( \theta + \id\right )$ be the corresponding eigenspaces. Every such $\theta$ thus defines a {\it symmetric decomposition} of $\g$, {\it i.e.} a triple $(\g,\h ,\p)$ such that
\begin{itemize}
\item{$\h \subset \g$ is a Lie subalgebra;}
\item{$\g = \h \oplus \p$ as $\K$-modules;}
\item{$[\h ,\p] \subseteq \p$ and $[\p,\p] \subseteq \h$. }
\end{itemize}
Any Lie subalgebra $\h$ of $\g$ that is the fixed point set of some involutive automorphism will be referred to as a {\it symmetrizing} subalgebra. If, in addition, $\g$ is semisimple then $\p$ must be the orthogonal complement of $\h$ in $\g$ with respect to the (non-degenerate) Killing form, and thus every given symmetrizing subalgebra $\h$ uniquely determines $\p$ and hence $\theta$.
In this case, we shall refer to $(\g, \h)$ as a {\it symmetric pair}.

A symmetric semisimple Lie algebra  $(\g, \theta)$ is said to be \textit{diagonal} if $\g = {\mf v} \oplus {\mf v}$ for some semisimple Lie algebra $\mf{v}$ and $\theta (x,y) = (y,x)$ for all $(x,y) \in \g$. A symmetric Lie algebra {\it splits} into symmetric subalgebras $(\g_i, \theta_i)_{i \in I}$ if $\g= \bigoplus_{i\in I} \g_i$ and the restrictions $\theta_{|\g_i} = \theta_i$ for all $i \in I$.
\begin{lemma}
\label{gvv}
Every symmetric semisimple Lie algebra $(\g, \theta)$ splits into a diagonal symmetric Lie algebra $(\g_d, \theta_d)$ and a collection of symmetric simple Lie subalgebras $(\g_i, \theta_i)_{i\in I}$.
\end{lemma}
A proof can be found in Chap. 8 of \cite{Helgasonbook}. Lemma \ref{gvv} allows for a complete classification of the symmetric semisimple Lie algebras; see \cite{Cartan, Helgasonbook}. It also follows that we have the following

\begin{lemma}
\label{pph}
Let $(\g, \theta)$ be a symmetric semisimple Lie algebra and let $\g = \h \oplus \p$ be the associated symmetric decomposition of $\g$. Then $\h$ is linearly generated by $[\p, \p]$.
\end{lemma}
\begin{proof}
By virtue of lemma \ref{gvv}, it suffices to prove this result on symmetric simple Lie algebras and on diagonal symmetric Lie algebras. Let us first assume that $\g$ is simple. The linear span of $[\p, \p]$ defines a non-trivial ideal in $\h$ and span$([\p, \p]) \oplus \p$ therefore defines a non-trivial ideal in $\g$. If we assume that $\g$ is simple, it immediately follows that span$([\p, \p]) = \h$. Suppose now that $(\g, \theta)$ is a diagonal symmetric Lie algebra, {\it i.e.} that there exists a semisimple Lie algebra ${\mathfrak v}$ such that $\g = {\mathfrak v} \oplus {\mathfrak v}$ and $\theta (x,y) = (y,x)$ for all $(x,y)\in \g$. In this case, we have a symmetric decomposition $\g = \h \oplus \p$, where $\h$ is the set of elements of the form $(x,x)$ for all $x \in {\mathfrak v}$, whereas $\p$ is the set of elements of the form $(x,-x)$ for all $x \in {\mathfrak v}$. We naturally have $[\p,\p]\subseteq \h$. Now, as ${\mathfrak v}$ is semisimple, it follows that for every $x \in {\mathfrak v}$, there exist $y, z \in {\mathfrak v}$ such that $x = [y,z]$. Then for all $(x,x) \in \h$, we have $(x,x) = ([y,z],[y,z]) = [(y,y),(z,z)] = [(y,-y),(z,-z)]$. But both $(y,-y)$ and $(z,-z)$ are in $\p$. \finproof 
\end{proof}

\section{Contractible QUEAs}
\label{SecII}

\subsection{Filtrations of the Universal Enveloping Algebra}
\label{FUEA}
Given a Lie algebra $\g$ over $\K$, its universal enveloping algebra $\U \g$ is defined as the quotient of the graded tensor algebra $T\g = \bigoplus_{n \geq 0} \g^{\otimes n}$ by the two-sided ideal ${\mathcal I}(\g)$ generated by the elements of the form $x \otimes y - y \otimes x - [x,y]$, for all $x,y \in \g$. This quotient constitutes a filtered $\K$-algebra, {\it i.e.} there exists an increasing sequence
\be
\label{filtreseq}
\{ 0 \} \subset F_0(\U \g) \subset \cdots \subset F_n(\U \g) \subset \cdots \subset \U \g \, ,
\ee
such that~\footnote{Although $F_n(\U \g) \cdot F_m(\U \g)$ is usually strictly contained in $F_{n+m}(\U \g)$, it linearly generates the latter.}
\be
\label{filtrecond}
\U \g = \bigcup_{n \geq 0} F_n (\U \g) \qquad \mbox{and} \qquad F_n(\U \g) \cdot F_m(\U \g) \subset F_{n+m}(\U \g) \, .
\ee
The elements of this sequence are, for all $n \in {\mathbb N_0}$,
\be
F_n(\U \g) = \bigoplus_{m=0}^n \g^{\otimes m} / {\mathcal I}(\g) \, .
\ee
In particular, $F_0(\U\g)=\K$ and $F_1(\U\g) = \K \oplus \g$. Let us identify $\g$ with its image under the canonical inclusion $\g\hookrightarrow \U(\g)$, and further write $x_1 \cdots x_n$ for the equivalence class of $x_1 \otimes \cdots \otimes x_n$. In this notation, $F_n(\U\g)$ is linearly generated by elements that can be written as words composed of at most $n$ symbols from $\g$.

We define the left action of $\g$ on $\g^{\otimes n}$ by extending the adjoint action $x\triangleright x'=\left[x,x'\right]$ of $\g$ on $\g$ as a derivation:
\be
\label{leftgaction0}
x \triangleright \left ( x_1 \otimes \cdots \otimes x_n \right ) = \sum_{i=1}^{n} x_1 \otimes \cdots \otimes [x,x_i] \otimes \cdots \otimes x_n  \in \g^{\otimes n} \, ,
\ee
for all $x, x_1, \dots, x_n \in \g$. 
In this way we endow $T\g$ with the structure of a left $\g$-module.
As the ideal ${\mathcal I}(\g)$ is stable under this action, the $F_n(\U \g)$ are also left $\g$-modules. We therefore have a filtration of $\U\g$ not only as a $\K$-algebra, but also as a left $\g$-module.

We will also need such a filtration on $\left (\U\g \right )^{\otimes 2}$. In fact, for all $m\in\mathbb N_0$, there is a $\K$-algebra filtration on the universal envelope $\U(\g^{\oplus m})$ of the Lie algebra $\g^{\oplus m}$, as defined above. If we endow $\g^{\oplus m}$ with the structure of a left $\g$-module according to
\be x\triangleright (x_1,\dots,x_m) := \left(\left[x,x_1\right],\dots, \left[x,x_m\right]\right)\,,\label{diagaction}\ee 
and extend this action to all of  $\U(\g^{\oplus m})$ as a derivation, then we have a filtration of $\U(\g^{\oplus m})$ as a left $\g$-module. But there is a natural isomorphism 
\be\label{UGm}\rho_m:\U(\g^{\oplus m})\isom \left(\U\g \right )^{\otimes m}\ee of $\K$-algebras (see e.g. \cite{Dixmier} section 2.2). This induces a left action of $\g$ on $\left(\U\g \right )^{\otimes m}$ and a filtration of $\left(\U\g \right )^{\otimes m}$ as a left $\g$-module. We write the elements of this filtration as $F_n\left( \left(\U\g \right )^{\otimes m}\right)$.

Given now any symmetric decomposition \be\g = \h \oplus \p,\ee there is an associated bifiltration $\left (F_{n,m}(\U\g) \right )_{n,m\in{\mathbb N_0}}$ of $\U\g$, {\it i.e.} a doubly increasing sequence
\be
F_{n,m}(\U\g) \subset  F_{n+1,m}(\U\g)  \qquad \mbox{and} \qquad   F_{n,m}(\U\g) \subset  F_{n,m+1}(\U\g) \, , 
\ee
such that
\be
\U\g = \bigcup_{n, m \geq 0} F_{n,m}(\U\g) \qquad \mbox{and}  \quad F_{n,m}(\U\g) \cdot F_{k,l}(\U\g) \subset F_{n+k,m+l}(\U\g) \, ,
\ee
for all $n,m,k,l \in {\mathbb N_0}$. The elements of this sequence are, for all $n,m \in {\mathbb N_0}$,
\be
F_{n,m}(\U\g) = \bigoplus_{p=0}^n \bigoplus_{q=0}^m  \mbox{Sym} \left ( \h^{\otimes p} \otimes \p^{\otimes q} \right ) /{\mathcal I} (\g) \, ,
\ee
where, for all $n \in {\mathbb N_0}$ and all $\K$-submodules $X_1, \dots X_n \subset \g$,
\be
\mbox{Sym} (X_1 \otimes \cdots \otimes X_n) = \bigoplus_{\sigma \in \Sigma_n} X_{\sigma(1)} \otimes \cdots \otimes X_{\sigma(n)}
\ee
is the direct sum over all permutations of submodules in the tensor product. Each $F_{n,m}(\U\g)$ is therefore the left $\h$-module linearly generated by elements of $\U\g$ that can be written as words containing at most $n$ symbols in $\h$ and at most $m$ symbols in $\p$. In particular, $F_{1,0}(\U\g) = \K \oplus \h$ and $F_{0,1}(\U\g)=\K\oplus \p$. We also have, for all $m,n \in {\mathbb N_0}$,
\be
\label{sumS}
F_{n,m}(\U\g) \subset F_{n+m}(\U\g) \qquad \mbox{and} \qquad F_n(\U\g) = \bigcup_{m=0}^n F_{n-m,m}(\U\g) \, .
\ee
In complete analogy with the $F_n((\U\g)^{\otimes m})$, we can construct bifiltrations $F_{n,p}((\U\g)^{\otimes m})$ of all the $m$-fold tensor products of $\U\g$.

\subsection{Symmetric tensors}
Let $S(\g)$ be the graded algebra associated to the filtration of $\U(\g)$ by setting, for all $n \in {\mathbb N}_0$,
\be S_n(\g) = F_n(\U \g) / F_{n-1} (\U \g) \qquad \mbox{and} \qquad S(\g) = \bigoplus_{n\geq 0} S_n(\g) \, . \ee
Since the $F_n(\U \g)$ are left $\g$-modules, so are the $S_n(\g)$. The {\it symmetrization map}, $\mbox{sym} : S(\g) \rightarrow \U\g$, defined by
\be
\mbox{sym}(x_1 \cdots x_n ) = \frac{1}{n!} \sum_{\sigma \in S_n} x_{\sigma(1)} \cdots x_{\sigma(n)}
\ee
for all $n \in {\mathbb N}_0$ and all $x_1, \dots, x_n \in \g$, constitutes an isomorphism of left $\g$-modules~\footnote{Recall that we assume $\K$ has characteristic zero.}. The image of a given $S_n(\g)$ through sym is the $\g$-module of symmetric tensors in $\g^{\otimes n}$.

If now $\g = \h \oplus \p$ is a symmetric decomposition, let
\be S_{m,n}(\g) =  F_{m,n}(\U\g) / F_{m+n-1}(\U\g) \, ,\ee
for all $m,n \in {\mathbb N}_0$. These obviously constitute left $\h$-modules. As such, they are isomorphic to the left $\h$-modules of symmetric tensors in the $\mbox{Sym}\left( \h^{\otimes m} \otimes \p^{\otimes n}\right)$, which are linearly generated by totally symmetric words with exactly $m$ symbols in $\h$ and exactly $n$ symbols in $\p$. Note that
these $\h$-modules are mixed under the left $\p$-action. Indeed, let $m, n \in {\mathbb N_0}$ be two non-negative integers and let $x\in S_{m,n}(\g)$. We have:
\begin{itemize}
 \item if $m>0$ and $n=0$, then $\p \triangleright x \in S_{m-1,n+1}(\g)$;
 \item if $m>0$ and $n>0$, then $\p \triangleright x \in S_{m+1,n-1}(\g) \oplus S_{m-1,n+1}(\g)$;
 \item if $m=0$ and $n>0$, then $\p \triangleright x \in S_{m+1,n-1}(\g)$.
\end{itemize}
This is better represented by the following diagram in $S_{m+n}(\g)$.
$$
\xymatrix{
\cdots \ar@{-->}[rd]^{\p\triangleright}& S_{m+1,n-1}(\g) \ar@{-->}[ld]^{\p\triangleright} \ar@{-->}[rd]^{\p\triangleright} \ar[d]^{\h \triangleright} & S_{m,n}(\g) \ar@{-->}[ld]^{\p \triangleright} \ar@{-->}[rd]^{ \p \triangleright} \ar[d]^{\h \triangleright }
& S_{m-1,n+1}(\g) \ar@{-->}[ld]^{\p \triangleright} \ar@{-->}[rd]^{\p \triangleright} \ar[d]^{\h \triangleright } & S_{m-2,n+2}(\g) \ar@{-->}[ld]^{ \p \triangleright}\ar[d]^{\h \triangleright } \ar@{-->}[rd]^{\p\triangleright}& \cdots \ar@{-->}[ld]^{\p\triangleright}\\
\cdots & S_{m+1,n-1}(\g) & S_{m,n}(\g) & S_{m-1,n+1}(\g) & S_{m-2,n+2}(\g) & \cdots}$$

Using the action (\ref{diagaction}) of $\g$ on $\g^{\oplus m}$ we have entirely analogous structures for $\g^{\oplus m}$ with
\be S_{n,p}(\g^{\oplus m}) = F_{n,p}(\U(\g^{\oplus m}))/ F_{n+p-1}(\U(\g^{\oplus m})) \, .\ee
In view of (\ref{UGm}), it follows that
\be \label{snp} S_{n,p}(\g^{\oplus m}) \cong F_{n,p}\left((\U \g)^{\otimes m})\right) / F_{n+p-1} \left((\U \g)^{\otimes m})\right) \ee
for all $n,p\in\mathbb N_0$. We shall therefore identify each $S_{n,p}(\g^{\oplus m})$ with the left $\h$-module of symmetric tensors on $(\U\g)^{\otimes m}$ containing exactly $n$ factors in $\h$ and $p$ in $\p$.

\subsection{Symmetric invariants and the restriction property}
For all $n, p \in {\mathbb N_0}$, let $S_n(\g\oplus \g)^\g$ be the set of $\g$-invariant elements of the left $\g$-module $S_n(\g \oplus \g)$ and let $S_{n,p}(\g \oplus \g)^\h$ denote the set of $\h$-invariant elements of the left $\h$-module $S_{n,p}(\g \oplus \g)$. We have the following two lemmas.
\begin{lemma}
\label{lemmafact}
Let $n$ and $p$ be positive integers. Every $x \in S_{n-p,p}(\g\oplus\g)^\h$ such that $\p \triangleright x \in S_{n-p+1,p-1}(\g\oplus\g)$ is in the linear span of $S_{n-p,0}(\g\oplus \g)^\g \, S_{0,p}(\g \oplus \g)^\h$.
\end{lemma}
\begin{proof}
Let $(h_i)_{i \in I}$ and $(p_j)_{j \in J}$ be ordered bases of $\h\oplus\h$ and $\p\oplus\p$ respectively. Every element $x \in S_{n-p,p}(\g\oplus\g)$ can be written as
\be
x = \sum_{i_1 \leq \cdots \leq i_{n-p}} \sum_{j_1 \leq \cdots \leq j_{p}} x_{i_{1} \dots i_{n-p} j_{1} \dots j_{p}} \, h_{i_1} \dots  h_{i_{n-p}} p_{j_1} \dots p_{j_{p}} \, , \nonumber
\ee
where, for all $i_{1}, \dots, i_{n-p} \in I$ and $j_{1}, \dots, j_{p} \in J$, $x_{i_{1} \dots i_{n-p} j_{1} \dots j_{p}} \in \K$. Then, omitting the ordered sums, we have
\be
\p \triangleright x = x_{i_{1} \dots i_{n-p} j_{1} \dots j_{p}} \, \left  [\p \triangleright \left ( h_{i_1} \dots h_{i_{n-p}} \right ) \, p_{j_1} \dots p_{j_{p}} + h_{i_1} \dots h_{i_{n-p}}  \, \p \triangleright \left ( p_{j_1} \dots p_{j_{p}} \right ) \right ] \, . \nonumber
\ee
Since $\left(\p \triangleright x\right) \cap S_{n-p-1,p+1}(\g\oplus\g) = \{0\}$, we have
\be\nn\p \triangleright \left ( x_{i_{1} \dots i_{n-p} j_{1} \dots j_{p}} \, h_{i_1} \dots h_{i_{n-p}} \right ) =0,\ee
for all $j_1 \leq \dots \leq j_{p} \in J$; it follows that this quantity is also invariant under $\left[ \p,\p\right]$ and hence, by lemma \ref{pph}, under $\h$. Thus it is actually $\g$-invariant. Introduce a basis $(y_k)_{k \in K}$ of the $\K$-module $S_{n-p,0}(\g\oplus\g)^\g$, so that we can write
\be
x_{i_{1} \dots i_{n-p} j_{1} \dots j_{p}} \, h_{i_1} \dots h_{i_{n-p}} = \sum_{k \in K} b_{k \, j_{1} \dots j_{p}} \, y_k\, , \nonumber
\ee
with $b_{k \, j_1 \dots j_p} \in \K$, for all $j_1 \leq \dots \leq j_{p} \in J$. Now, as $x$ is $\h$-invariant, we also have
\be
\h \triangleright x = b_{k \, j_{1} \dots j_{p}} \,  y_k\, \h \triangleright \left (p_{j_1} \dots p_{j_{p}} \right ) = 0\, . \nonumber
\ee
This yields $\h \triangleright \left ( b_{k \, j_{1} \dots j_{p}} \, p_{j_1} \dots p_{j_{p}} \right ) = 0$, for all $k \in K$. Introduce a basis $(z_l)_{l \in L}$ for the $\K$-module $S_{0,p}(\g\oplus\g)^\h$, so that we can write, for all $k \in K$,
\be
 b_{k \, j_{1} \dots j_{p}} \, p_{j_1} \dots p_{j_{p}} = \sum_{l \in L} a_{kl} z_l \, , \nonumber
\ee
with $a_{kl} \in \K$ for all $k \in K$ and $l \in L$. Now, $x$ can be rewritten as 
\be
x = \sum_{k \in K} \sum_{l \in L} a_{kl} \, y_k\, z_l \, , \nonumber
\ee
with $y_k \in S_{n-p,0}(\g\oplus\g)^\g$ for all $k \in K$ and $z_l \in S_{0,p}(\g\oplus\g)^\h$ for all $l \in L$. \finproof
\end{proof}

\noindent Let us now restrict our attention to the class of symmetric Lie algebras encompassed by the following
\begin{definition}
\label{RP}
We say that a symmetric semisimple Lie algebra $(\g,\theta)$ with associated symmetric decomposition $\g = \h \oplus \p$ is of \emph{restrictive type} (or has the \emph{restriction property}) if and only if for all $p \in {\mathbb N_0}$, the projection from $\g$ to $\p$ maps $S_{p}(\g \oplus \g)^\g$ onto $S_{0,p}(\g \oplus\g)^\h$.
\end{definition}
This restriction property will be sufficient to allow us to prove a refined version of Whitehead's lemma in the next section. Note that it is similar to the so-called \emph{surjection property} -- namely that the restriction from $\g$ to $\p$ maps $S(\g)^\g$ onto $S(\p)^\h$ -- which is known to hold for all classical symmetric Lie algebras \cite{Helgason} and which has proven useful in a number of contexts \cite{Burstalletal}. In our case we have, at least,

\begin{lemma}
\label{restriction}
If a symmetric semisimple Lie algebra splits (as in lemma \ref{gvv}), in such a way that its simple factors are drawn only from the following classical families of simple symmetric Lie algebras:
\be \mbox{AI}_{n>2} : (\mf{su}(n),\mf{so}(n))_{n>2}\, , \quad \mbox{AII}_{n} : (\mf{su}(2n),\mf{sp}(2n))_{n \in {\mathbb N^*}} \, , \quad  \mbox{BDI}_{n>2,1} : (\mf{so}(n+1), \mf{so}(n))_{n>2}\, ,\nn\ee
then it is of restrictive type.
\end{lemma}
\begin{proof}
See appendix. \finproof
\end{proof}

\subsection{Contractible homomorphisms of $\K[[h]]$-modules}
Let $\K[[h]]$ denote the $\K$-algebra of formal power series in $h$ with coefficients in the field $\K$ and let $\U\g[[h]]$ be the $\U\g$-algebra of formal power series in $h$ with coefficients in $\U\g$. We have a natural $\K$-algebra monomorphism $i : \U\g \hookrightarrow \U\g[[h]]$. There is also an epimorphism of $\K$-algebras $j : \U\g[[h]] \twoheadrightarrow \U\g$ such that $j \circ i = \id$ on $\U\g$. We shall therefore identify $\U\g$ with its image $i(\U\g) \subset \U\g[[h]]$. We shall also consider complete $\K[[h]]$-modules and it is assumed that the tensor products considered from now on are completed in the $h$-adic topology. In this subsection, we further assume that $\g = \h \oplus \p$ is a symmetric decomposition.

\begin{definition}
Let $p \in {\mathbb Z}$, $m\in {\mathbb N_0}$ be integers. An element $x$ of $(\U \g)^{\otimes m} [[h]]$ is $(p,\p)$-contractible if and only if there exists a collection $(x_{n})_{n\in{\mathbb N_0}}$ of elements of $(\U\g)^{\otimes m}$ such that,
\be
x = \sum_{n\geq 0} h^n \, x_{n} 
\ee
and, for all $n \in {\mathbb N_0}$, there exists $l(n) \in {\mathbb N_0}$ such that $x_{n} \in F_{l(n),n+p}\left ( (\U \g)^{\otimes m} \right )$.
\end{definition}
Similarly, a subset $X \subset (\U \g)^{\otimes m} [[h]]$ is $(p,\p)$-contractible if all its elements are, according to the previous definition. Note that for the sake of simplicity, we shall refer to $(0,\p)$-contractible elements or sets as $\p$-contractible. Let us now define the notion of contractibility for $\K[[h]]$-module homomorphisms in $\mbox{Hom}\left (\U\g ^{\otimes m}[[h]], (\U \g)^{\otimes n} [[h]] \right )$.
\begin{definition}
Let $r, s\in {\mathbb N_0}$ and $p \in {\mathbb Z}$ be integers. A homomorphism of $\K[[h]]$-modules $\phi :   (\U\g) ^{\otimes r}[[h]] \rightarrow (\U \g)^{\otimes s} [[h]]$ is $\p$-contractible if and only if, for all $n,m \in {\mathbb N_0}$, $\phi (F_{n,m} (\U\g ^{\otimes r}))$ is $(m,\p)$-contractible as a subset.
\end{definition}
Let us emphasize that for every $\p$-contractible $\K[[h]]$-module homomorphism $\phi : (\U\g) ^{\otimes r}[[h]] \rightarrow (\U \g)^{\otimes s} [[h]]$, there exists a collection $(\varphi_{n})_{n\in{\mathbb N_0}}$ of $\K[[h]]$-module homomorphisms $\varphi_{n} : (\U\g) ^{\otimes r}[[h]] \rightarrow (\U\g)^{\otimes s}[[h]]$ such that
\be
\phi = \sum_{n\geq 0} h^n \, \varphi_{n} \ee
and, for all $n,m,p \in {\mathbb N_0}$, there exists $l(n) \in {\mathbb N_0}$ such that $\varphi_{n} \left ( F_{m,p}((\U\g) ^{\otimes r} ) \right ) \subseteq F_{l(n),n+p}\left ((\U \g)^{\otimes s} \right )$. The following two lemmas shall be useful in the next sections.
\begin{lemma}
\label{lemmeutile2}
Let $\phi$ and $\psi$ be two $\p$-contractible homomorphisms of $\K[[h]]$-modules. Then the $\K[[h]]$-module homomorphism $\phi\circ\psi$ is $\p$-contractible.
\end{lemma}

\begin{proof}
We have
\be
\phi = \sum_{n \geq 0} h^n \, \varphi_n \qquad \mbox{and} \qquad \psi = \sum_{n\geq 0} h^n \, \psi_n \, , \nonumber
\ee
with, for all $n,m,p \in {\mathbb N_0}$, $\varphi_n(F_{m,p}) \subseteq F_{*,n+p}$, and $\psi_n(F_{m,p}) \subseteq F_{*,n+p}$. For the sake of simplicity we shall omit the arguments of the bifiltration and denote by $*$ the integer $l(n)$ whose existence is guaranteed by the definition of contractibility. We thus have
\be
\phi \circ \psi = \sum_{n \geq 0} \sum_{m \geq 0} h^{n+m} \, \varphi_n \circ \psi_m = \sum_{n \geq 0} h^{n} \, \sum_{m = 0}^n  \varphi_m \circ \psi_{n-m} \, ,\nonumber
\ee
with, for all $l, m,n,p  \in {\mathbb N_0}$, $\varphi_{m} \circ \psi_{n-m} (F_{l,p}) \subseteq \varphi_{m} (F_{*,n-m+p}) \subseteq F_{*,n+p}$. \finproof
\end{proof}

The following holds for the inverse.
\begin{lemma}
\label{lemmeutile3}
Let $\phi$ be a $\p$-contractible homomorphism of $\K[[h]]$-modules, congruent with $\id \mod h$. Then the $K[[h]]$-module homomorphism $\phi^{-1} = \id \mod h$ is $\p$-contractible.
\end{lemma}

\begin{proof}
We shall construct
\be
\phi^{-1} = \sum_{n \geq 0} h^n \, \varphi_n\, , \nonumber
\ee
by recursion on the order in $h$, by demanding that $\phi \circ \phi^{-1} = \id$. At leading order, we have $\varphi_0 = \id$ and therefore $\varphi_0 (F_{m,p}) \subseteq F_{m,p}$, for all $m,p \in {\mathbb N_0}$. Let us assume that we have a polynomial $\phi^{-1}_n$ of degree $n>0$ such that
\be
\phi \circ \phi^{-1}_n - \id = q \mod h^{n+1} \, . \nonumber
\ee
Assuming that $\phi_n^{-1}$ is $\p$-contractible, we have by lemma \ref{lemmeutile2} that $\phi \circ \phi^{-1}_n$ is $\p$-contractible, as $\phi$ is $\p$-contractible by assumption. Therefore, $q(F_{m,p}) \subseteq F_{*,n+1+p}$. Now, to complete the recursion, we have to find $\varphi_{n+1}$ such that
\be
\phi \circ \left (\phi^{-1}_n + h^{n+1} \, \varphi_{n+1} \right ) - \id = 0 \mod h^{n+2} \, . \nonumber
\ee
This is achieved by taking $\varphi_{n+1} = -q$. We thus have $\varphi_{n+1}(F_{m,p}) \subseteq F_{*,n+1+p}$. \finproof
\end{proof}

\noindent Finally, when $\phi$ is not only a $\K[[h]]$-module homomorphism but also a $\K[[h]]$-algebra homomorphism, we have the following useful lemma.
\begin{lemma}
\label{lemmeutile1}
Let $\phi : (\U \g)^{\otimes s}[[h]] \rightarrow (\U \g)^{\otimes t} [[h]]$ be a homomorphism of $\K[[h]]$-algebras. It is $\p$-contractible if and only if $\phi \left ( F_{1,0}((\U \g)^{\otimes s}) \right )$ is $(0,\p)$-contractible and $\phi \left ( F_{0,1}((\U \g)^{\otimes s}) \right )$ is $(1,\p)$-contractible.
\end{lemma}

\begin{proof}
If $\phi$ is $\p$-contractible, it follows from the definition that, in particular, $\phi \left ( F_{1,0} \right )$ is $(0,\p)$-contractible and $\phi \left ( F_{0,1} \right )$ is $(1,\p)$-contractible. Now, assuming that $\phi \left ( F_{1,0}\right )$ is $(0,\p)$-contractible and $\phi \left ( F_{0,1}\right )$ is $(1,\p)$-contractible, we want to prove that, for all $m,p \in {\mathbb N_0}$, $\phi \left ( F_{m,p} \right )$ is $(p,\p)$-contractible. We proceed by recursion on $m$ and $p$. We have assumed the result for $m=1$ and $p=0$, as well as for $m=0$ and $p=1$. Suppose that, for some $m,p \in {\mathbb N_0}$, we have proven that, for all $m'<m$, $p'<p$ and $n\in{\mathbb N}_0$, there exists $l \in{\mathbb N_0}$ such that $\varphi_n \left ( F_{m',p'} \right ) \subseteq F_{l,n+p'}$. Then, for all $n \in {\mathbb N_0}$,
\bea
\varphi_n \left ( F_{m,p+1} ((\U \g)^{\otimes s}) \right ) &=& \varphi_n \left ( \bigoplus_{k=0}^m \bigoplus_{l=0}^p \, \mbox{span} \quad  F_{k,l} \cdot F_{0,1} \cdot F_{m-k,p-l} \right ) \nonumber \\
&=& \bigoplus_{k=0}^m \bigoplus_{l=0}^p \, \mbox{span} \quad \sum_{\sigma \in C_3(n)} \varphi_{\sigma_1} \left (F_{k,l} \right )  \cdot \varphi_{\sigma_2} \left ( F_{0,1} \right ) \cdot \varphi_{\sigma_3} \left (  F_{m-k,p-l} \right ) \nonumber \\
&\subseteq& \bigoplus_{k=0}^m \bigoplus_{l=0}^p \, \mbox{span}_{\sigma \in C_3(n)} \quad F_{*,\sigma_1+l} \cdot F_{*, \sigma_2+1} \cdot F_{*,\sigma_3+p-l} = F_{*,n+p+1} \nonumber \, ,
\eea
where, for all $X\subseteq (\U\g)^{\otimes s}$, span~$X$ denotes the $\K$-module linearly generated by $X$ and $C_3(n)$ is the set $\{\sigma=(\sigma_1,\sigma_2,\sigma_3)\in \mathbb N_0^3 : \sum_{i=1}^3 \sigma_i = n\}$  of weak 3-compositions of $n$.
Similarly, we have
\bea
\varphi_n \left ( F_{m+1,p} ((\U \g)^{\otimes s}) \right ) &=& \varphi_n \left ( \bigoplus_{k=0}^m \bigoplus_{l=0}^p \, \mbox{span} \quad  F_{k,l} \cdot F_{1,0} \cdot F_{m-k,p-l} \right ) \nonumber \\
&=&  \bigoplus_{k=0}^m \bigoplus_{l=0}^p \, \mbox{span} \quad \sum_{\sigma \in C_3(n)} \varphi_{\sigma_1} \left (F_{k,l} \right )  \cdot \varphi_{\sigma_2} \left ( F_{1,0} \right ) \cdot \varphi_{\sigma_3} \left (  F_{m-k,p-l} \right ) \nonumber \\
&\subseteq& \bigoplus_{k=0}^m \bigoplus_{l=0}^p \, \mbox{span}_{\sigma \in C_3(n)} \quad F_{*,\sigma_1+l} \cdot F_{*, \sigma_2} \cdot F_{*,\sigma_3+p-l} = F_{*,n+p} \nonumber \, ,
\eea
for all $n \in {\mathbb N_0}$. \finproof
\end{proof}

\subsection{Contractible deformation Hopf algebras}
We recall that $\U(\g)$ possesses a natural cocommutative Hopf algebra structure, whose coproduct is the algebra homomorphism $\Delta_0 : \U \g \rightarrow \U \g \otimes \U \g$ defined by 
$\Delta_0(x) = x \otimes 1 + 1 \otimes x$ for all $x \in \g$, and whose counit and antipode are specified by $\eps_0(1) = 1$ and $S_0(1) =1$. We refer to this as the \textit{undeformed} Hopf algebra structure.

Given the notion of contractibility introduced in the preceding subsections, it is natural to specialize the usual notion of a quantization -- {\it i.e.} a deformation -- of a universal enveloping algebra, as follows.

\begin{definition}\label{pdefHopf}
Let $(\g, \theta)$ be a symmetric Lie algebra, with symmetric decomposition $\g = \h \oplus \p$. A $\p$-contractible deformation $(\U_h \g, \cdot_h, \Delta_h, \epsilon_h, S_h)$ of the Hopf algebra $(\U\g, \cdot, \Delta_0, \epsilon_0, S_0)$ is a topological Hopf algebra such that
\begin{itemize}
\item{there exists a $\K[[h]]$-module isomorphism $\eta : \U_h\g \isom \U\g[[h]]$;}
\item{$\mu_h := \eta \circ \left ( \cdot_h \right ) \circ \left (\eta^{-1} \otimes \eta^{-1} \right ) = \cdot \mod h$ and $\mu_h$ is $\p$-contractible;}
\item{$\Deltat_h := \left ( \eta \otimes \eta \right ) \circ \Delta_h \circ \eta^{-1}= \Delta_0 \mod h$ and $\Deltat_h$ is $\p$-contractible;}
\item{$\tilde{S}_h := \eta \circ S_h \circ \eta^{-1} = S_0 \mod h$ and $\tilde{S}_h$ is $\p$-contractible;}
\item{$\tilde{\epsilon}_h = \epsilon_h \circ \eta^{-1} = \epsilon_0 \mod h$ and $\tilde{\epsilon}_h$ is $\p$-contractible.}
\end{itemize}
\end{definition}
This definition can be naturally restricted to bialgebras and algebras.

\section{On the cohomology of associative and Lie algebras}
\label{SecIII}

\subsection{The Hochschild cohomology}
Let $A$ be a $\K$-algebra. For any $(A,A)$-bimodule $(M,\triangleright, \triangleleft)$ and all $n \in {\mathbb N_0}^*$, we define the $(A,A)$-bimodule of $n$-cochains $C^n(A, M) = \mbox{Hom} (A^{\otimes n}, M)$. We also set $C^0(A, M) = M$. To each cochain module $C^n(A,M)$, we associate a coboundary operator, {\it i.e.} a derivation operator $\delta_n : C^n(A, M) \longrightarrow C^{n+1}(A, M)$, by setting, for all $f \in C^n(A, M)$,
\bea
\delta_n f \left (x_1, \dots , x_{n+1} \right ) &=& x_1 \triangleright f \left (x_2, \dots, \hat{x}_i, \dots, x_{n+1} \right ) + \sum_{i=1}^{n} \left ( -1 \right )^{i} \, f \left (x_1, \dots, x_i x_{i+1}, \dots, x_{n+1} \right ) \nonumber \\
&& + \left ( -1 \right )^{n+1} \, f \left ( x_1, \dots, x_n \right ) \triangleleft x_{n+1} 
\label{dnHoch}
\eea
for all $x_1, \dots, x_{n+1} \in A$. One can check that $\delta_n \circ \delta_{n+1} =0$ for all $n$. Therefore, the $(C^n, \delta_n)$ thus defined constitute a cochain complex. It is known as the Hochschild or {\it standard} complex \cite{Hoch} -- see also \cite{CEbook} or \cite{Weibel}. An element of the $(A,A)$-bimodule $Z^n(A, M) = \ker \delta_n \subset C^n(A, M)$ is called an $n$-{\it cocycle}, while an element of the $(A,A)$-bimodule $B^n(A, M) = \mbox{im} \, \delta_{n-1} \subset C^n(A, M)$ is called an $n$-{\it coboundary}. As usual, the quotient
\be
HH^n (A, M) = Z^n (A, M) / B^n (A, M)
\ee
defines the $n$-th cohomology module of $A$ with coefficients in $M$. In the next section, we shall be particularly interested in the Hochschild cohomology of the universal enveloping algebra of a given Lie algebra $\g$, {\it i.e.} $A = \U \g$, with coefficients in $M= \U \g$. The latter trivially constitutes a $(\U \g, \U\g)$-bimodule with the multiplication $\cdot$ of $\U\g$ as left and right $\U\g$-action. Concerning the Hochschild cohomology we will need the following result -- see for example theorem 6.1.8 in \cite{Chari}.
\begin{lemma}
\label{HH0}
Let $\g$ be a semisimple Lie algebra over $\mathbb K$. Then, $HH^2(\U\g,\U\g) =0$.
\end{lemma}

\subsection{The Chevalley-Eilenberg cohomology}
Let $\g$ be a Lie algebra over $\mathbb K$ and $\left (M, \triangleright \right )$ a left $\g$-module. For all $n \in {\mathbb N_0}^*$, we define the left $\g$-module of $n$-cochains $C^n(\g, M) = \mbox{Hom} (\wedge^n \g, M)$, with left $\g$-action
\be
\left ( x \triangleright f \right ) \left ( x_1, \dots, x_n \right ) = x \triangleright \left ( f(x_1, \dots, x_n) \right ) - \sum_{i=1}^n f \left (x_1, \dots, [x, x_i], \dots, x_n \right ) \, ,
\ee
for all $f \in C^n(\g, M)$ and all $x, x_1, \dots, x_n \in \g$. We also set $C^0(\g, M) = M$ with its natural left $\g$-module structure. To each cochain module $C^n(\g,M)$, we associate a coboundary operator, {\it i.e.} a derivation operator $d_n : C^n(\g, M) \longrightarrow C^{n+1}(\g, M)$, by setting, for all $f \in C^n(\g, M)$,
\bea
d_n f \left (x_1, \dots , x_{n+1} \right ) &=& \sum_{i=1}^{n+1} \left ( -1 \right )^{i+1} \, x_i \triangleright f \left (x_1, \dots, \hat{x}_i, \dots, x_{n+1} \right ) \nonumber \\
\label{dn}
&& + \sum_{1 \leq i \leq j \leq n+1} \left ( -1 \right )^{i+j} \, f \left (\left [ x_i, x_j \right ], x_1, \dots , \hat{x}_i, \dots, \hat{x}_j, \dots, x_{n+1} \right ) \quad
\eea
for all $x_1, \dots, x_{n+1} \in \g$. In (\ref{dn}), hatted quantities are omitted and $\triangleright$ denotes the left $\g$-action on $M$. One can check that $d_n \circ d_{n+1} =0$ for all $n$. Therefore, the $(C^n, d_n)$ thus defined constitute a cochain complex. It is known as the Chevalley-Eilenberg complex \cite{CE}, -- see also \cite{CEbook} or \cite{Weibel}. An element of $Z^n(\g, M) = \ker d_n \subset C^n(\g, M)$ is called an $n$-{\it cocycle}, while an element of $B^n(\g, M) = \mbox{im} \, d_{n-1} \subset C^n(\g, M)$ is called an $n$-{\it coboundary}. As usual, the quotient
\be
H^n (\g, M) = Z^n (\g, M) / B^n (\g, M)
\ee
defines the $n$-th cohomology module of $\g$ with coefficients in $M$. One can check that, for all $n\in \mathbb N_0$, $Z^n(\g,M)$, $B^n(\g,M)$ and $H^n(\g,M)$ naturally inherit the left $\g$-module structure of $C^n(\g,M)$, as for all $n \in \mathbb N_0$,
\be
d \left (x \triangleright f \right ) = x \triangleright df \, ,
\ee
for all $f \in C^n(\g,M)$ and all $x \in \g$. An important result about the Chevalley-Eilenberg cohomology of Lie algebras concerns finite dimensional complex semisimple Lie algebras. It is known as Whitehead's lemma.
\begin{lemma}
\label{Whitehead}
Let $\g$ be a semisimple Lie algebra over $\mathbb K$. If M is any finite-dimensional left $\g$-module, then $H^1(\g,M) = H^2(\g,M) = 0$.
\end{lemma}
A proof of this result can be found, for instance, in section 7.8 of \cite{Weibel}.

\subsection{Contractible Chevalley-Eilenberg cohomology}
\label{ContractCEcohom}
In the next section, we will be mostly interested in the module $M= \U \g \otimes \U \g$, with the left $\g$-action induced by (\ref{diagaction}) and (\ref{UGm}), {\it i.e.} 
\be g \triangleright x = [\Delta_0 (g), x]\, ,\ee
for all $g \in \g$ and all $x \in \U\g\otimes \U\g$. In particular, we shall need a refinement of Whitehead's lemma, in the case of symmetric semisimple Lie algebras of restrictive type, taking into account the possible $\p$-contractibility of the generating cocycles of $Z^*(\g,\U\g \otimes \U\g)$. For all $m,n \in {\mathbb N_0}$, we therefore define $C^n_{m,\p}(\g, \U\g\otimes \U\g)$ as the set of $(m,\p)$-contractible $n$-cochains, by which we mean the set of $n$-cochains $f \in C^n(\g,\U\g \otimes \U\g)$, such that, for all $0 \leq p \leq n$, $f \left ((\wedge^{n-p} \h) \wedge (\wedge^p \p) \right ) \subseteq F_{l,m+p}(\U\g \otimes \U\g)$, for some $l \in {\mathbb N_0}$. Defining similarly, $Z^n_{m,\p}(\g, \U\g \otimes \U\g)= \ker d_n \cap C^n_{m,\p}(\g, \U\g\otimes \U\g)$ and $B^n_{m,\p}(\g,\U\g \otimes \U\g) = d_{n-1} C^{n-1}_{m,\p}(\g, \U\g\otimes \U\g)$ as the modules of the $(m,\p)$-contractible $n$-cocycles and of the $n$-coboundaries of $(m,\p)$-contractible $n-1$-cochains, respectively, we can define the $n$-th $(m,\p)$-contractible cohomology module as
\be
H^n_{m,\p}(\g, \U\g \otimes \U\g)= Z^n_{m,\p}(\g, \U\g\otimes \U\g) / B^n_{m,\p}(\g,\U\g \otimes \U\g) \, .
\ee
It is worth emphasizing that these cohomology modules generally differ from the usual ones $H^n(\g, \U\g \otimes \U\g)$. Consider for instance a case for which $H^1(\g, \U\g \otimes \U\g)=0$. We have that every $1$-cocycle in $Z^1(\g, \U\g\otimes \U\g)$, and therefore every cocycle $f \in Z^1_{m,\p}(\g, \U\g\otimes \U\g)$, is the coboundary of an element $x \in \U\g \otimes \U\g$. However, although the considered $f$ is $(m,\p)$-contractible, it may be that it can only be obtained as the coboundary of an element $x \in \U\g \otimes \U\g$ that does not belong to any $F_{*,m}(\U\g \otimes \U\g)$, thus yielding a non-trivial cohomology class in $H^1_{m,\p}(\g, \U\g \otimes \U\g)$. When $\g$ is a symmetric semisimple Lie algebra of restrictive type, we nonetheless establish the following lemma concerning the first $(m,\p)$-contractible cohomology module $H^1_{m,\p}(\g, \U\g \otimes \U\g)$.
\begin{lemma}
\label{WhiteheadRef}
Let $(\g, \theta)$ be a symmetric semisimple Lie algebra of restrictive type over $\mathbb K$ and let $\g =\h \oplus \p$ be the associated symmetric decomposition of $\g$. We have $H^1_{m,\p}(\g, \U\g \otimes \U\g)=0$, for all $m\in{\mathbb N_0}$.
\end{lemma}
\begin{proof}
Let $m\in {\mathbb N_0}$ be a positive integer. We have to prove that every $(m,\p)$-contractible $1$-cocycle $f\in Z^1_{m,\p}(\g, \U\g\otimes \U\g)$ is the coboundary of an element $\alpha \in F_{l,m}(\U\g \otimes \U\g)$, for some $l \in {\mathbb N_0}$. From lemma \ref{Whitehead}, there exists an $x \in \U\g \otimes \U\g$ such that $f=d_0 x$. All we have to prove is that we can always find a left $\g$-invariant $y \in \left (\U\g \otimes \U\g \right )^\g$, such that $x=y$ modulo $F_{l,m}(\U\g\otimes \U\g)$ for some $l \in {\mathbb N_0}$. Then, we can check that for $\alpha = x - y \in F_{l,m}(\U\g \otimes \U\g)$, we have
\be
d_0 \alpha = d_0 \left (x- y \right ) = d_0 x = f\, . \nonumber
\ee
In view of (\ref{snp}), we can first expand $x$ into its components in the left $\g$-modules isomorphic to the $S_n(\g \oplus \g)$, for all $n \in {\mathbb N_0}$. Up to the isomorphism of left $\g$-modules, which we shall omit here, we have $x = \sum_{n\geq0} x_n$ where, for all $n \in {\mathbb N_0}$, $x_n \in S_n(\g\oplus\g)$. Similarly, we can further decompose each $S_n(\g \oplus \g)$ into the left $\h$-modules $S_{n-p,p}(\g\oplus \g)$, with $0\leq p\leq n$, and, accordingly, each $x_n$. We are now going to construct the desired $y \in \left (\U\g \otimes \U\g \right )^\g$ by recursion, submodule by submodule. 
If $x_n = 0$ for all $n > m$, we can set $y=0$ and we are done. So, suppose that there exists an $n>m$ such that $x_n \neq 0$ and let $x_{0,n}$ be the component of $x_n$ in $S_{0,n}(\g\oplus \g)$. If $x_{0,n}$ vanishes, we can skip to the component of $x_n$ in $S_{1,n-1}(\g\oplus\g)$. Otherwise, we are going to prove that there exists a $\g$-invariant $y_{n,0} \in S_{n}(\g\oplus\g)^\g$, such that the component of $x_n-y_{n,0}$ in $S_{0,n}(\g\oplus\g)$ vanishes.
From $f$ being $(m,\p)$-contractible, we know that
\be
\label{hinv}
f(\h) = d_0 x (\h) = \h \triangleright \left ( x_n + \sum_{n' \neq n} x_{n'} \right )\subseteq F_{l,m}(\U\g \otimes \U\g) \, ,
\ee
for some $l \in {\mathbb N_0}$. Therefore, since the $S_{m,p}(\g\oplus\g)$ are left $\h$-modules, we have $\h \triangleright x_{0,n} =0$. Since $\g$ has the restriction property, definition \ref{RP}, it follows that the $\h$-invariant tensor $x_{0,n} \in S_{0,n}(\g \oplus\g)^\h$ is the restriction to $\p$ of a $\g$-invariant tensor $y_{n,0} \in S_{n}(\g\oplus\g)^\g$. Now consider $x_n-y_{n,0}$. By construction, it has no component in $S_{0,n}(\g\oplus\g)$. If $n-1\leq m$, we set $y_n = y_{n,0}$ and skip to another $\g$-module $S_{n'>m}(\g\oplus\g)$ where $x$ has a non-vanishing component, if any. Otherwise, let $0\leq k < n-m$ and assume that we have found $y_{n,k} \in S_{n}(\g\oplus\g)^\g$, such that $x_n-y_{n,k}$ has vanishing component in all the $S_{n-p,p}(\g\oplus\g)$ with $p \geq n-k > m$. We are going to prove that there exists $y_{n,k+1} \in S_{n}(\g\oplus\g)^\g$ such that $x_n-y_{n,k+1}$ has vanishing component in all the $S_{n-p,p}(\g\oplus\g)$ with $p \geq n-k-1$. To do so, let $x_{k+1,n-k-1}$ be the component of $x_n-y_{n,k}$ in $S_{k+1,n-k-1}(\g\oplus\g)$. If it is zero, we set $y_{n,k+1}=y_{n,k}$. Otherwise, note that from (\ref{hinv}), we have $\h \triangleright x_{k+1,n-k-1} =0$.
But the $(m,\p)$-contractibility of $f$ also implies that
\be
f(\p) = d_0 x (\p) = \p \triangleright \left ( x_n-y_{n,k} + \sum_{n' \neq n} x_{n'} \right )\subseteq F_{l,m+1}(\U\g \otimes \U\g) \, , \nonumber
\ee
from which it follows that $\p \triangleright x_{k+1,n-k-1} \in S_{k+2,n-k-2}(\g\oplus\g)$. According to lemma \ref{lemmafact}, we can write $x_{k+1,n-k-1}= \sum_{i, j} a_{ij} \, w_i \, z_j$, with $a_{ij} \in \K$, $w_i \in S_{k+1,0}(\g\oplus\g)^\g$ and $z_j \in S_{0,n-k-1}(\g\oplus\g)^\h$. Since $\g$ has the restriction property, all the $z_j$ are the restrictions to $\p$ of $\g$-invariant elements $\zeta_j \in S_{n-k-1}(\g\oplus\g)^\g$. Now, set $y_{n,k+1}=y_{n,k} + \sum_{i, j} a_{ij} \, w_i \, \zeta_j$. It is obvious that $y_{n,k+1} \in S_{n}(\g\oplus\g)^\g$ and, by construction, $x_{n}-y_{n,k+1}$ has no component in all the $S_{n-p,p}(\g\oplus\g)$, with $p \geq n-k-1$.
The recursion goes on until we have $y_{n,n-m} \in S_{n}(\g\oplus\g)^\g$ such that $x_n-y_{n,n-m}$ has vanishing components in all the $S_{n-p,p}(\g\oplus\g)$, with $p>m$. We therefore set $y_n = y_{n,n-m}$. By repeating this a finite number of times~\footnote{It is rather obvious that $x$ has non-vanishing components in a finite number of submodules $S_n(\g \oplus \g)$, as there always exists an $l \in {\mathbb N}$ such that $x \in F_{l}(\U\g\otimes\U\g)$.}, in all the $S_{n'>m}(\g\oplus\g)$ in which $x$ has non-vanishing components, we obtain the desired $y = \sum_{n \geq 0} y_n$. \finproof
\end{proof}

\section{Rigidity theorems}
\label{SecIV}

\subsection{Contractible algebra isomorphisms}
\begin{proposition}
\label{Propisom}
Let $\g$ be a semisimple Lie algebra over $\K$ and let $\h$ be a symmetrizing Lie subalgebra with orthogonal complement $\p$ in $\g$. Then, for every $\p$-contractible deformation algebra $(\U_h \g, \cdot_h)$ of $(\U \g, \cdot)$, there exists a $\p$-contractible isomorphism of $\K[[h]]$-algebras $(\U_h \g, \cdot_h) \isom (\U \g[[h]], \cdot)$, that is congruent with $\id \mod h$.
\end{proposition}

\begin{proof}
By definition, there exists a $\K[[h]]$-module isomorphism $\eta : \U_h \g \isom \U \g [[h]]$. The latter defines a $\K[[h]]$-algebra isomorphism between $(\U_h\g, \cdot_h)$ and $(\U\g[[h]], \mu_h)$, where $\mu_h := \eta \circ \left ( \cdot_h \right ) \circ \left (\eta^{-1} \otimes \eta^{-1} \right ) = \cdot \mod h$. If we found a $\p$-contractible $\K[[h]]$-algebra automorphism
\be
\label{phiautom}
\phi : (\U\g[[h]], \mu_h) \isom (\U\g[[h]], \cdot) \, ,
\ee
we would prove the proposition as $\phi \circ \eta$ would constitute the desired $\K[[h]]$-algebra isomorphism from $(\U_h\g, \cdot_h)$ to $(\U\g[[h]], \cdot)$. Let $\phi$ be a $\K[[h]]$-module automorphism on $\U\g[[h]]$. The condition for such an automorphism to be the $\K[[h]]$-algebra automorphism (\ref{phiautom}) is
\be
\label{muh}
\mu_h = \phi^{-1} \circ ( \cdot )  \circ \left ( \phi \otimes \phi \right ) \, .
\ee
Let us construct
\be
\phi = \sum_{n\geq 0} h^n \, \varphi_{n} \, ,
\ee
order by order in $h$. At leading order, we have $\mu_0 = \cdot$ and we can take $\varphi_0= \id \in \mbox{Hom}(\U \g[[h]],\U\g[[h]])$. We thus have $\varphi_0 (F_{m,p}(\U\g)) \subseteq F_{m,p}(\U\g)$, for all $m,p \in {\mathbb N_0}$. Suppose now that we have found a polynomial of degree $n>0$,
\be
\phi_n = \sum_{m=0}^{n} h^m \, \varphi_{m} \, ,
\ee
such that
\be
\label{muhn}
\mu_h - \phi_n^{-1} \circ ( \cdot )  \circ \left ( \phi_n \otimes \phi_n \right ) = h^{n+1} r \mod h^{n+2} \, ,
\ee
where $\phi_n^{-1}$ denotes the exact inverse series of $\phi_n$ defined by $\phi_n \circ \phi_n^{-1} = \id$ and $r \in \mbox{Hom}(\U\g \otimes \U\g[[h]], \U\g[[h]])$. We assume that $\phi_n$ is $\p$-contractible. Therefore, $\left ( \cdot \right ) \circ \left ( \phi_n \otimes \phi_n \right )$ is $\p$-contractible. By lemma \ref{lemmeutile3}, $\phi_n^{-1}$ is $\p$-contractible and, by lemma \ref{lemmeutile2}, $\phi_n^{-1} \circ \left ( \cdot \right ) \circ \left (\phi_n \otimes \phi_n \right )$ is $\p$-contractible. By definition of a $\p$-contractible deformation algebra, we know that $\mu_h$ is $\p$-contractible.
It therefore follows from (\ref{muhn}) at order $h^{n+1}$ that $r(F_{m,p}(\U\g \otimes \U\g)) \subseteq F_{*,n+1+p}(\U \g)$, for all $m,p \in {\mathbb N_0}$. From the associativity of $\mu_h$, we deduce that $r$ is a 2-cocycle in the Hochschild complex,
\be
\delta_2 r = 0 \, .
\ee
As $\g$ is semisimple, it follows from lemma \ref{HH0} that its second Hochschild cohomology module $HH^2(\U \g, \U \g)$ is empty, so that $r$ is a coboundary. We thus have $r=\delta_1 \beta$, for some $\beta \in \mbox{Hom}(\U\g[[h]],\U\g[[h]])$. But we know that, in particular, $r(F_{2,0}(\U\g \otimes \U\g)) \subseteq F_{*,n+1}(\U \g)$ and $r(F_{1,1}(\U\g \otimes \U\g)) \subseteq F_{*,n+2}(\U \g)$. It follows that $\beta$ can be consistently chosen so that $\beta(F_{1,0}(\U\g)) \subseteq F_{*,n+1}(\U\g)$ and $\beta(F_{0,1}(\U\g)) \subseteq F_{*,n+2}(\U\g)$. To complete the recursion, we have to solve
\be
\mu_h = \left (\phi_n^{-1} - h^{n+1} \varphi_{n+1} \mod h^{n+2} \right ) \circ \left [ \left (\phi_n +h^{n+1} \varphi_{n+1} \right ) \cdot \left (\phi_n +h^{n+1} \varphi_{n+1} \right ) \right ] \mod h^{n+2} \nonumber
\ee
that is
\be
\delta_1 \varphi_{n+1} = r \, .
\ee
This equation can be solved by taking $\varphi_{n+1}= -\beta$, which implies that $\varphi_{n+1}(F_{1,0}(\U\g)) \subseteq F_{*,n+1}(\U \g)$ and $\varphi_{n+1}(F_{0,1}(\U\g)) \subseteq F_{*,n+2}(\U \g)$. The proposition then follows from lemma \ref{lemmeutile1}. \finproof
\end{proof}

\subsection{Contractible twisting for symmetric semisimple Lie algebras}
\begin{proposition}
\label{Twistprop}
Let $(\g, \theta)$ be a symmetric semisimple Lie algebra over $\K$ having the restriction property, and let $\g= \h \oplus \p$ be the associated symmetric decomposition of $\g$. Every $\p$-contractible deformation $(\U_h \g, \Delta, \epsilon, S)$ of the Hopf algebra $(\U \g, \Delta_0, \epsilon_0, S_0)$ is isomorphic, as a Hopf algebra over $\K[[h]]$, to a twist of $(\U \g, \Delta_0, \epsilon_0, S_0)$ by a $\p$-contractible invertible element $\F \in \U \g \otimes \U \g [[h]]$, congruent with $1\otimes 1 \mod h$.
\end{proposition}

\begin{proof}
We consider the composite map
\be
\Deltat : \, \, \U \g [[h]] \isom \U_h \g \stackrel{\Delta}{\longrightarrow} \U_h \g \otimes \U_h \g \isom \U \g \otimes \U \g [[h]] \, ,
\ee
where the existence of a $\p$-contractible isomorphism of $\K[[h]]$-algebras $\phi$ follows from proposition \ref{Propisom}. As $\phi$ is an algebra isomorphism, the composite map $\Deltat$ is an algebra homomorphism. By repeated use of lemma \ref{lemmeutile2}, one can show that it is $\p$-contractible. 
Now, we want to prove that there exists a $\p$-contractible and invertible element $\F \in \U \g \otimes \U \g[[h]]$, such that $\F = 1 \otimes 1 \mod h$ and
\be
\Deltat = \F \Delta_0 \F^{-1} \, .
\ee
We shall proceed by recursion on the order in $h$. To first order, we have, by construction
\be
\Deltat = \Delta_0 \mod h
\ee
and we can take $\F = 1\otimes 1 \mod h$. We thus have $\F_{|h=0} \in F_{0,0}(\U\g \otimes \U\g)$. Suppose now that we have found a polynomial $\F_n \in \U \g \otimes \U \g [h]$ of degree $n$, 
\be
\F_n = \sum_{m=0}^n h^m \, f_{m} \, ,
\ee
such that
\be
\label{Deltaxi}
\Deltat - \F_n \Delta_0 \F_n^{-1} = h^{n+1} \xi \mod h^{n+2} \, ,
\ee
where $\F_n^{-1} \in \U \g \otimes \U \g [[h]]$ is the formal inverse of $\F$ in the sense that $\F^{-1} \F = 1$ and $\xi \in \mbox{Hom} (\U \g[[h]], \U \g \otimes \U \g[[h]])$. We assume that $\F_n$ is $\p$-contractible, {\it i.e.} for all $n\in{\mathbb N_0}$, $f_n \in F_{*,n}(\U\g \otimes \U\g)$. Since $\Deltat$ is $\p$-contractible, we deduce that $\xi (F_{1,0}(\U\g)) \subseteq F_{*,n+1}(\U\g\otimes\U\g)$ and $\xi(F_{0,1}(\U\g) \subseteq F_{*,n+2}(\U\g)$.
It follows from (\ref{Deltaxi}) that, for all $X, Y \in \g$, we have
\be
\label{dxi1}
\left ( \Deltat - \F_n \Delta_0 \F_n^{-1}  \right ) ([X,Y]) = h^{n+1} \xi ([X,Y]) \mod h^{n+2} \, ,
\ee
on one hand and, on the other hand, since $\Deltat$ is an algebra homomorphism,
\bea
\label{dxi2}
\left ( \Deltat - \F_n \Delta_0 \F_n^{-1}  \right ) ([X,Y]) &=& \left [\Deltat X, \Deltat Y \right ] - \F_n \Delta_0 ([X,Y]) \F_n^{-1} \nonumber \\
&=& h^{n+1} \left ( \left [\Delta_0 X, \xi (Y) \right ] +  \left [\xi(X) , \Delta_0 Y \right ]  \right ) \mod h^{n+2} .\,  \,
\eea
Equating (\ref{dxi1}) and (\ref{dxi2}), we finally get
\be
d_1 \xi = 0 \, .
\ee
The map $\xi$ is thus a 1-cocycle of $Z^1(\g, \U\g\otimes\U\g)$ in the sense of the Chevalley-Eilenberg complex~\footnote{By rewriting (\ref{dxi1}-\ref{dxi2}) for the associative product of two arbitrary elements in $\U\g$, we also show that $\xi$ is a 1-cocycle in the sense of the Hochschild complex. This indeed provides a unique continuation of $\xi$ from $\g$ to $\U\g$ as a derivation.}. As $\g$ is semisimple, it follows from lemma \ref{Whitehead} that the cohomology module $H^1 (\g, \U \g \otimes \U \g)$ is empty. We therefore conclude that $\xi$ is a coboundary. 
But we know that $\xi(F_{0,1}(\U\g)) \subseteq F_{*,n+2}(\U\g \otimes \U\g)$ and $\xi (F_{1,0}(\U\g)) \subseteq F_{*,n+1}(\U\g \otimes \U\g)$, so that $\xi$ is an $(n+1,\p)$-contractible 1-cocycle in the contractible Chevalley-Eilenberg complex defined in subsection \ref{ContractCEcohom}. As $\g$ is of restrictive type, it follows from lemma \ref{WhiteheadRef}, that $H_{n+1,\p}^1(\g, \U\g\otimes\U\g)=0$, so that $\xi$ is the coboundary of an $(n+1,\p)$-contractible element in $\U\g\otimes\U\g$, {\it i.e.} there exists an $\alpha \in F_{*,n+1}(\U\g\otimes\U\g)$ such that $\xi = d_0 \alpha = \delta_0 \alpha$.
In order to complete the recursion, we have to find an $f_{n+1} \in \U \g \otimes \U \g$ such that
\be
\Deltat - \left ( \F_n + h^{n+1} f_{(n+1)} \right ) \Delta_0 \left (\F_n^{-1} - h^{n+1} f_{(n+1)} \mod h^{n+2} \right ) = 0 \mod h^{n+2} \, .
\ee
Expanding the above equation to order $h^{n+1}$ yields
\be
\delta_0 f_{n+1} + \xi = 0 \, .
\ee
This equation can then be solved by choosing $f_{n+1} = -\alpha \in F_{*,n+1}(\U\g\otimes \U\g)$. \finproof
\end{proof}

\subsection{Contractible triangular quasi-Hopf algebra structure}
Recall, \cite{Chari, Majidbook}, that the notion of twisting extends to quasi-triangular quasi-Hopf (qtqH) algebras: given a Hopf algebra $\mathcal H = (A, \cdot , \Delta, S, \epsilon, 1)$ equipped with qtqH algebra structure $(\R , \Phi)$, the twisted Hopf algebra $\mathcal H^\F = (A, \cdot, \F \Delta \F^{-1}, S, \epsilon, 1)$ has the qtqH algebra structure $(\R^\F, \Phi^\F)$, where 
\be \R^\F = \F_{21} \R \F^{-1} \quad \mbox{and} \quad  \Phi^\F = \F_{12} \cdot \left (\Delta \otimes \id \right ) (\F) \cdot \Phi \cdot \left (\id \otimes \Delta \right ) ( \F^{-1}) \cdot \F_{23}^{-1}
 \, .\ee
Naturally, we say that a qtqH structure $(\R,\Phi)$ on the QUEA $\U_h(\g)$ is $\p$-contractible with respect to a symmetric decomposition $\g=\h\oplus \p$ if and only if $\R$ and $\Phi$ are $\p$-contractible as elements of, respectively, $(\U_h\g)^{\otimes 2}$ and $(\U_h\g)^{\otimes 3}$. It then follows from the definitions above that

\begin{proposition}
For any QUEA $\U_h \g$ and any symmetric decomposition $\g=\h\oplus \p$, if $(\R, \Phi)$ is a $\p$-contractible qtqH algebra structure for $\U_h(\g)$ and $\F\in (\U_h\g)^{\otimes 2}$ is a $\p$-contractible twist then $(\R^\F,\Phi^\F)$ is a $\p$-contractible qtqH algebra structure on the twisted Hopf algebra $\U_h(\g)^\F$.
\end{proposition}

\noindent Combining this with propositions \ref{Propisom} and \ref{Twistprop}, we have that every $\p$-contractible qtqH algebra structure on a QUEA $\U_h \g$ can be obtained, via a change of basis and twist, from some $\p$-contractible qtqH algebra structure on the undeformed envelope $\U\g$. In particular, starting from the trivial triangular structure $(\R= 1 \otimes 1, \Phi = 1 \otimes 1 \otimes 1)$ on $\U\g$, which is obviously $\p$-contractible, we have

\begin{corollary}
\label{CorRPhi}
Every $\p$-contractible deformation Hopf algebra $(\U_h \g, \Delta, \epsilon, S)$ based on a symmetric semisimple Lie algebra of restrictive type with symmetric decomposition $\g= \h \oplus \p$, admits a $\p$-contractible triangular quasi-Hopf algebra structure.
\end{corollary}

\begin{proof}
Explicitly, by propositions \ref{Propisom} and \ref{Twistprop}, there exist a $\p$-contractible invertible element $\F \in \U\g \otimes \U\g [[h]]$ and a $\p$-contractible $\K[[h]]$-algebra isomorphism $\phi$, such that 
\be
\Delta = \left ( \phi^{-1} \otimes \phi^{-1} \right ) \circ \F \Delta_0  \F^{-1} \circ \phi \, .\nonumber
\ee
Defining
\bea
\R &:=& \phi^{-1} \otimes \phi^{-1} \left ( \F_{21} \F^{-1} \right ) \, , \\
\Phi &:=& \phi^{-1} \otimes \phi^{-1} \otimes \phi^{-1} \left ( \F_{12} \cdot \left (\Delta_0 \otimes \id \right ) (\F) \cdot \left (\id \otimes \Delta_0 \right ) ( \F^{-1}) \cdot \F_{23}^{-1} \right ) \, 
\eea
provides the required structure.
\finproof
\end{proof}

\noindent It is natural to ask whether any qtqH algebra structures on the envelope $\U \g$ other than the trivial one are $\p$-contractible. In section \ref{SecVI}, we provide an example for which this is the case and one for which it is not.

\section{Twists and $\kappa$-deformations}
\label{SecV}
We can now finally turn to the objects in which we are really interested in this paper: those deformed enveloping algebras of non-semisimple Lie algebras that are obtained by a certain contraction procedure modelled on that used in \cite{Celeghini2,Celeghini,Lukierski} to obtain the $\kappa$-deformation of Poincar\'e. The notion of $\p$-contractibilty introduced in the previous sections is formulated with this type of contraction in mind, as we now discuss.

Recall first that if $\g = \h \oplus \p$ is a symmetric decomposition of a Lie algebra $\g$, a standard procedure known as {\it In\"onu-Wigner contraction}, \cite{IWContractions}, consists in contracting the submodule $\p$ by means of a one-parameter family of linear automorphisms of the form 
\be\Lambda_t = \pi_\h + t \, \pi_\p\,,\ee 
where $\pi_\h : \g \twoheadrightarrow \h$ and $\pi_\p : \g \twoheadrightarrow \p$ denote the linear projections from $\g$ to $\h$ and $\p$ respectively and $t \in (0,1]$.
For all $t \in (0,1]$, the image of $\g$ by the automorphism $\Lambda_t^{-1}$ is the symmetric semisimple Lie algebra $\g_t$, isomorphic to  $\g=\h\oplus\p$ as a $\K$-module, with Lie bracket 
\be [X,Y]_t = \Lambda_t^{-1} \left ( [\Lambda_t( X ) , \Lambda_t ( Y ) ] \right )\ee 
for all $X,Y\in \g$. It has the property that
\be
[\h, \h]_t \subset \h \, , \qquad [\h, \p]_t \subset \p \, , \quad \mbox{and} \quad [\p, \p]_t \subset t^2 \h \, . 
\ee
so in the limit $t\to 0$ one obtains a Lie algebra $\g_0$, isomorphic to $\g=\h\oplus \g$ as a $\K$-module, whose Lie bracket $[,]_0 = \lim_{t \to 0} [,]_t$ obeys
\be
[\h, \h]_0 \subset \h \, , \qquad [\h, \p]_0 \subset \p \, , \quad \mbox{and} \quad [\p, \p]_0 = \{ 0\} \, .
\ee
The submodule $\p$ is therefore an abelian ideal in $\g_0$. The undeformed Hopf algebra structure defined in section \ref{FUEA} is preserved as $t$ tends to zero. There is thus a natural undeformed Hopf algebra structure $(\U (\g_0), \Delta_0, S_0, \epsilon_0)$ on the envelope $\U(\g_0)$ of the contracted Lie algebra. 

We may extend $\Lambda_t$ over $\U\g[[h]]$ as a $\K[[h]]$-algebra homomorphism. Further, by means of the $\K[[h]]$-module isomorphism $\eta$ of definition \ref{pdefHopf}, we can regard $\Lambda_t$ as a map $\U_h\g\rightarrow \U_h\g$ on any QUEA $\U_h\g$. This specifies how every element of the latter is to be rescaled in the contraction limit. 

The relevance of the definition of $\p$-contractibility from section \ref{SecII} is then contained in the following

\begin{defnprop} 
\label{Defkappa}
Let $(\g, \theta)$ be a symmetric semisimple Lie algebra with symmetric decomposition $\g = \h \oplus \p$ and let $(\U_h\g, \Delta_h, S_h, \epsilon_h)$ be a deformation of the Hopf algebra $(\U \g, \Delta_0, S_0, \epsilon_0)$. For all $t \in (0,1]$, set
\be
\Delta_t = (\Lambda_t^{-1} \otimes \Lambda_t^{-1} ) \circ \Delta_{t/ \kappa} \circ \Lambda_t \, , \qquad S_t = \Lambda_t^{-1} \circ S_{t/\kappa} \circ \Lambda_t \, \quad \mbox{and} \quad \epsilon_t = \epsilon_{t/ \kappa} \circ \Lambda_t \, ,
\ee
where $1/\kappa = h/t$ is the rescaled deformation parameter. Then the limit of $(\U_{t/\kappa}(\g_t), \Delta_t,  S_{t}, \epsilon_{t})$ as $t \to 0$ exists if and only if $(\U_h\g, \Delta_h, S_h, \epsilon_h)$ is $\p$-contractible.  We write QUEAs so obtained as $\left (\U_\kappa(\g_0), \Delta_\kappa, S_\kappa, \epsilon_\kappa \right)$, and refer to them as \emph{$\kappa$-contractions} of  $(\U_h\g, \Delta_h, S_h, \epsilon_h)$ and as \emph{$\kappa$-deformations} of $(\U (\g_0), \Delta_0, S_0, \epsilon_0)$.
\end{defnprop}
\begin{proof}
Let $r,s \in {\mathbb N}$ and let $\phi : (\U\g) ^{\otimes r}[[h]] \rightarrow (\U \g)^{\otimes s} [[h]]$ be a homomorphism of $\K[[h]]$-modules. We want to prove that $\phi_t = (\Lambda_t^{-1})^{\otimes s} \circ \phi \circ (\Lambda_t)^{\otimes r}$ has a finite limit when $t \to 0$ if and only if $\phi$ is $\p$-contractible. First assume that $\phi$ is $\p$-contractible; then from lemma \ref{lemmeutile1}, there exists a collection $(\varphi_{n})_{n\in{\mathbb N_0}}$ of $\K[[h]]$-module homomorphisms $\varphi_{n} : (\U\g) ^{\otimes r}[[h]] \rightarrow (\U\g)^{\otimes s}[[h]]$ such that
\be
\phi = \sum_{n\geq 0} h^n \, \varphi_{n} \ee
and, for all $n,m,p \in {\mathbb N_0}$, there exists $l \in {\mathbb N_0}$ such that $\varphi_{n} \left ( F_{m,p}((\U\g) ^{\otimes r} ) \right ) \subseteq F_{l,n+p}\left ((\U \g)^{\otimes s} \right )$. We thus have, for all $n, m, p \in {\mathbb N_0}$,
\bea
h^n \, (\Lambda_t^{-1})^{\otimes s} \circ \varphi_n \circ (\Lambda_t)^{\otimes r} \left ( S_{m,p}(\g ^{\oplus r} ) \right ) &=& \kappa^{-n} t^{n+p} \, (\Lambda_t^{-1})^{\otimes s} \circ \varphi_n \left (S_{m,p}(\g^{\oplus r} ) \right ) \nonumber \\
&\subseteq& \kappa^{-n} t^{n+p} \,  (\Lambda_t^{-1})^{\otimes s} \left (F_{l,n+p}((\U\g) ^{\otimes s} ) \right ) \nonumber \\ 
&=& \kappa^{-n} t^{n+p} \, O(t^{-(n+p)}) \, F_{l,n+p}((\U\g) ^{\otimes s} ) \nonumber \\
&=& \kappa^{-n} O(1) \, F_{l,n+p}((\U\g) ^{\otimes s} ) \, .\nonumber
\eea
This obviously has a finite limit when $t \to 0$ and so does $\phi_t$. Conversely, one sees that if $\phi$ is not $\p$-contractible, $\phi_t$ diverges at least as $t^{-1}$. \finproof
\end{proof}

\noindent It is worth emphasizing that the $\kappa$-contractions are a restricted subclass among the possible contractions that can be performed on QUE algebras: one could also, for example, consider contractions where the deformation parameter $h$ is not rescaled in the limit.

Finally, we can state our main result concerning twists and $\kappa$-deformations:
\begin{theorem}
\label{thF}
If a deformation Hopf algebra $(\U_\kappa (\g_0), \Delta_\kappa, S_\kappa, \epsilon_\kappa)$ is the $\kappa$-contraction of a QUEA of a symmetric Lie algebra $(\g,\theta)$ having the restriction property, then it is isomorphic, as a Hopf algebra over $\K[[h]]$, to a twist of the undeformed Hopf algebra $(\U (\g_0), \Delta_0, S_0, \epsilon_0)$ by an invertible element $\F_0 \in \U_\kappa (\g_0) \otimes \U_\kappa (\g_0) [[1/\kappa]]$ congruent with $1\otimes 1$ modulo $1/\kappa$. Thus, in particular, it admits a triangular quasi-Hopf algebra structure.
\end{theorem}
\begin{proof}
By proposition \ref{Defkappa}, proposition \ref{Twistprop} applies. By arguing as in the proof of \ref{Defkappa}, we have that if $\F$ is the $\p$-contractible twist element of proposition \ref{Twistprop}, then 
\be \F_0 = \lim_{t \to 0} \,\,(\Lambda_t^{-1} \otimes \Lambda_t^{-1} )(\F) \ee
is well-defined. By construction, this is the twist we seek. The existence of a triangular quasi-Hopf algebra structure then follows from corollary \ref{CorRPhi}. \finproof
\end{proof}

\section{Examples: $\kappa$-Poincar\'e in 3 and 4 dimensions}
\label{SecVI}
We now turn to explicit examples. Let $\K = \mathbb C$, and consider the symmetric decomposition
\be \mf{so}(n+1) = \mf{so}(n) \oplus \p_n \, , \qquad n>2 \, , \ee
whose In\"onu-Wigner contraction of course yields the Lie algebra $\mf{iso}(n)$ of the complexified Euclidean group in $n$ dimensions, $ISO(n,\mathbb C)$. By lemma \ref{restriction}, this decomposition is of restrictive type. Thus, the results above will apply to any $\p_n$-contractible deformation algebra $\U_h(\mf{so}(n+1))$. Finding such deformations is itself a non-trivial task. In the cases $n=3, 4$, this was achieved in \cite{Celeghini, Lukierski}~\footnote{Note that although the $\kappa$-Poincar\'e algebra exists in arbitrary dimension \cite{kPoind}, to the authors' knowledge it has only explicitly been shown to arise as a $\kappa$-contraction for $n \leq 4$.}, yielding the $\kappa$-deformations $\U_\kappa(\mf{iso}(3))$ and $\U_\kappa(\mf{iso}(4))$. These can be written in terms of the generators
\be M_{ij} = - M_{ji}\, , \quad N_i \, , \quad P_i, \quad P_0=E \, , \ee
for all $1 \leq i, j \leq n-1$ and $n=3,4$. The algebra is then given by
\be 
 \left[ M_{ij}, P_k \right ] = \delta_{k[i} P_{j]} \ee
\be \left [N_i, P_j \right ] = \delta_{ij} \, \kappa \sinh \left (\frac{E}{\kappa} \right ) \, , \qquad  \left [ N_i, E\right ] = P_i \, , \ee
\be\left [ N_i, N_j \right ] = - M_{ij} \cosh \left (\frac{E}{\kappa} \right ) + \frac{1}{4\kappa^2} \left (\vec P \ip \vec P M_{ij} + P_k P_{[i} M_{j]k} \right ) \, ,
\ee
for all $1 \leq i, j, k , l \leq n-1$. The coproduct is given by
\bea
\Delta_\kappa (E) &=& E \otimes 1 + 1 \otimes E \, , \label{coalg1} \\ 
\Delta_\kappa (P_i) &=& P_i \otimes e^{\frac{E}{2\kappa}} + e^{-\frac{E}{2\kappa}} \otimes P_i \, , \\
\Delta_\kappa (N_i) &=& N_i \otimes e^{\frac{E}{2\kappa}} + e^{-\frac{E}{2\kappa}} \otimes N_i + \frac{1}{2\kappa} \left (P_j \otimes e^{\frac{E}{2\kappa}} M_{ij}  - e^{-\frac{E}{2\kappa}} M_{ij} \otimes P_j \right ) \, ,\\
\Delta_\kappa (M_{ij}) &=& M_{ij} \otimes 1+ 1 \otimes M_{ij} \, ,\label{coalg2}
\eea
and the antipode by
\be S_\kappa (P_\mu) = -P_\mu, \quad S_\kappa(M_{ij}) = - M_{ij}, \quad S_\kappa(N_i) = - N_i + \frac{d}{2\kappa} P_i.\ee
The counit map is undeformed, $\eps(M_{ij}) = \eps(N_i) = \eps(P_\mu) = 0$, for all $0\leq \mu \leq n-1$. It follows from the results presented in the previous sections that both $\U_\kappa (\mathfrak{iso}(3))$ and $\U_\kappa(\mf{iso}(4))$ possess a triangular quasi-Hopf algebra structure. This provides, for $n=3, 4$, a proof of the results already anticipated in \cite{YZ}. The result that $\U_\kappa(\mf{iso}(3))$ and $\U_\kappa(\mf{iso}(4))$ are twist equivalent to the corresponding undeformed UEAs should not be confused with other statements that exist in the literature, \cite{Govindarajan}, concerning twists and $\kappa$-deformed Minkowski space-time, which involve enlarged algebras that include the dilatation generator.

One can also understand the existence of the quasi-triangular Hopf algebra structure of $\U_\kappa(\mf{iso}(3))$ exhibited in \cite{Celeghini} in the context of the results above: in the case $n=3$ only, $\U(\mf{so}(n+1))$ possesses a $\p$-contractible quadratic Casimir, namely $h \, {\bf t}:= h \, \epsilon_{ijk} M_{ij} P_k$, and, by twisting $(\R, \Phi)=(\exp(h {\bf t}), \Phi_{\mbox{\scriptsize KZ}})$ by means of the $\p$-contractible twist of proposition \ref{Twistprop}, one obtains a $\p$-contractible quasi-triangular Hopf algebra structure. For $n \neq 3$, there is no classical r-matrix obeying the classical Yang-Baxter equation \cite{Zakrzewski,YZ} and therefore no quasi-triangular Hopf algebra structure.

As for versions of the $\kappa$-deformed Poincar\'e algebra in higher and lower space-time dimensions, a consistent definition was first given in \cite{kPoind}. The main idea is that the four dimensional case is generic enough that the $1+d$-dimensional case can be obtained by simply extending or truncating the range of the spatial indices from $1, \dots, 3$ to $1, \dots, d$. It is reasonable to think that the twist obtained in the four dimensional case can be similarly extended to arbitrary dimensions, thus extending to all dimensions the existence of a triangular quasi-Hopf algebra structure on the $\kappa$-deformation of the Poincar\'e algebra. A rigorous proof of this fact would however require further investigation. Note also that there exists another, conceptually distinct, construction of $\kappa$-Poincar\'e, namely as a bicrossproduct \cite{bicross} -- see also \cite{MajidSchroers, Freidel}. It would be interesting to understand the above results from this point of view.

\section{Conclusion and outlook}
We have constructed triangular quasi-Hopf algebra structures on a family of non-semisimple QUEAs -- including the $\kappa$-deformations of the Euclidean and Poincar\'e algebras in three and four dimensions -- which are obtained by $\kappa$-contraction of QUEAs based on semisimple Lie algebras of a certain class. The representations of each of these $\kappa$-deformed UEAs therefore constitute a tensor category. The construction of these triangular quasi-Hopf algebra structures crucially involves twisting by a cochain twist which fails, in general, to obey the usual 2-cocycle condition, thus yielding a non-trivial coassociator. The proof of the existence of this twist relies on the vanishing of a certain cohomology class in a refined version of the Chevalley-Eilenberg complex, which, in turn, is guaranteed by the restriction property of definition \ref{RP}. Although this constitutes a sufficient condition, we do not expect that it is necessary. In particular, we expect that the $\kappa$-deformation of $\U (\mathfrak{sl}(2))$ admits a triangular quasi-Hopf algebra structure \cite{YZ3}, but a proof of this statement would obviously require a refinement of the arguments used here so as to circumvent the obstructions arising in this case -- cf. the appendix. Such a refinement could, for instance, rely on a further symmetry property of the $\p$-contractible Chevalley-Eilenberg cohomology of $\mathfrak{sl}(2)$.

The twist-equivalence of theorem \ref{thF} also guarantees the existence of {\it quasi}-triangular quasi-Hopf algebra structures on every $\kappa$-deformation whose underlying symmetric Lie algebra is of restrictive type, thus giving rise to a genuinely braided representation theory ({\it i.e.} a quasi-tensor category): by virtue of Drinfel'd's results \cite{Drinfeld}, one need only pick a quadratic Casimir of the contracted Lie algebra to construct the qtqH algebra structure $(\R_{\mbox{\scriptsize KZ}}, \Phi_{\mbox{\scriptsize KZ}})$ on the undeformed UEA. It is then natural to ask under what circumstances $\kappa$-contracted QUEAs admit a quasi-triangular {\it Hopf} algebra structure. A cohomological approach to this question -- cf. \cite{Donin} -- would certainly prove helpful.

\newpage
\section*{Appendix: proof of lemma \ref{restriction}}
In this appendix, we provide a proof of lemma \ref{restriction}. Let $(\g,\theta)$ be a symmetric semisimple Lie algebra obeying the conditions of the lemma. If $\g = \h \oplus \p$ is the associated symmetric decomposition of $\g$, we want to prove that, for all $p \in {\mathbb N}$, the projection from $\g$ to $\p$ maps $S_{p}(\g \oplus \g)^\g$ onto $S_{0,p}(\g \oplus \g)^\h$. The isomorphism of left $\g$-modules (\ref{UGm}) induces a similar isomorphism $S(\g\oplus\g) \cong S(\g) \otimes S(\g)$ at the level of the symmetric algebras, from which it follows that 
\be S_m(\g \oplus \g) \cong \bigoplus_{k=0}^m S_k(\g) \otimes S_{m-k}(\g) \, ,\ee
for all $m \in {\mathbb N}$. We thus have a decomposition of $S(\g \oplus \g)$ into the $\g$-submodules isomorphic to $S_k(\g) \otimes S_{m-k}(\g)$. There is an analogous decomposition of $S_{0,m}(\g\oplus\g)$ into $\h$-submodules isomorphic to $S_{0,k}(\g) \otimes S_{0,m-k}(\g) = S_k(\p) \otimes S_{m-k}(\p)$. It therefore suffices to show that, for all $k,\ell\in\mathbb N$, the restriction map induces a surjection
\be \left( S_k(\g) \otimes S_\ell(\g) \right)^\g \twoheadrightarrow 
    \left( S_{k}(\p) \otimes S_{\ell}(\p)\right)^\h \ee
Identifying $\g\cong \g^*$, and in particular $\p\cong \p^*$, by means of the Killing form, an element $d\in S_{k}(\p) \otimes S_{\ell}(\p)$ can be regarded as a $(k+\ell)$-linear map
\be \p\times \dots \times \p \rightarrow \mathbb K; \quad (X,\dots ,Y) \mapsto d(X,\dots,Y) \ee
that is symmetric in its first $k$ and final $\ell$ slots. In view of the polarization formulae, such maps are in bijection with polynomials of two variables in $\p$, according to
\be p_{(d)} (X,Y) = d(\underset{k}{\underbrace{X,\dots,X}},\underset{\ell}{\underbrace{Y,\dots,Y}} ) \, .\ee
These polynomials are $(k,\ell)$-homogeneous, by which we mean that they are homogeneous of degree $k$ with respect to their first argument and of degree $\ell$ with respect to their second argument. We denote by $\K_{k,\ell}[\p,\p]$ the left $\h$-module of $(k, \ell)$-homogeneous polynomials on $\p$. Then for all $k, \ell \in {\mathbb N}$, $\left ( S_{k}(\p) \otimes S_{\ell}(\p) \right )^{\h}$ is in bijection with the submodule of $\h$-invariant $(k, \ell)$-homogeneous polynomials of $\K_{k,\ell}[\p,\p]^\h$. Similarly, $\left( S_k(\g) \otimes S_\ell(\g) \right)^\g$ is in bijection with $\K_{k,\ell}[\g, \g]^\g$. Therefore, it suffices to show that the restriction map from $\g$ to $\p$ maps $\K_{k,\ell}[\g, \g]^\g$ onto $\K_{k,\ell}[\p,\p]^\h$. By virtue of lemma \ref{gvv}, it will be sufficient to consider separately the cases of diagonal symmetric Lie algebras and of the symmetric simple Lie algebras listed in \ref{restriction}.

We recall that a diagonal symmetric Lie algebra is a pair $(\g, \theta)$, where $\g = {\mathfrak v} \oplus {\mathfrak v}$, for some semisimple Lie algebra ${\mathfrak v}$, and $\theta$ is the involutive automorphism of Lie algebras defined by $\theta (x,y) = (y,x)$, for all $(x,y) \in \g$. We thus have $\g = \h \oplus \p$, where $\h$ is the set elements of $\g$ of the form $(x,x)$, whereas $\p$ is the set of elements of $\g$ of the form $(x,-x)$, for $x \in \mathfrak v$. We are first going to prove that $\K_{k,\ell}[\p, \p]^\h \cong \K_{k,\ell}[\mathfrak{v}, \mathfrak{v}]^{\mathfrak{v}}$. Let $p \in \K_{k,\ell}[\p, \p]$ be a polynomial. For all $X,Y\in \p$, we have
\be p(X,Y)= p((x,-x),(y,-y)) = \tilde{p}(x,y) \, ,\ee
for some $x,y \in \mathfrak v$. The left $\h$-action on $\p$ induces a left $\h$-action on $\p \times \p$, given, for all $h \in \h$ and all $X,Y\in \p$, by
\be h \triangleright (X,Y) = (z,z) \triangleright ((x,-x),(y,-y)) = ((z \triangleright x, - z \triangleright x),(z \triangleright y, -z \triangleright y)) \, ,\ee
for some $x,y \in \mathfrak v$ and some $z \in \mathfrak v$; from which it obviously follows that $\tilde{p}$ is $\mathfrak v$-invariant if and only if $p$ is $\h$-invariant. Now, we are going to prove that the restriction map is a surjection from $\K_{k,\ell}[\g,\g]^\g$ onto $\K_{k,\ell}[\mathfrak{v}, \mathfrak{v}]^{\mathfrak{v}}$. Let $p \in \K_{k,\ell}[\g,\g]^\g$ be a $\g$-invariant polynomial on $\g$. The left $\g$-action on $\g \oplus \g$ is given, for all $g \in \g$ and all $X,Y \in \g$, by
\be g \triangleright (X,Y) = (g_1, g_2) \triangleright ((x_1,x_2),(y_1,y_2)) = ((g_1 \triangleright x_1, g_2 \triangleright x_2),(g_1 \triangleright y_1, g_2 \triangleright y_2))\, ,\ee
for some $g_1, g_2 \in \mathfrak v$ and some $x_1, x_2, y_1, y_2 \in \mathfrak v$. As one can always choose $g_1$ and $g_2$ independently, it follows that in order for $p$ to be $\g$-invariant, there must be a polynomial $f : \K \times \K \rightarrow \K$ and two $\mathfrak v$-invariant polynomials $p_1, p_2 \in \K_{k,\ell}[\mathfrak v, \mathfrak v]^{\mathfrak v}$ such that
\be p((x_1,x_2),(y_1,y_2)) = f(p_1(x_1, y_1), p_2(x_2, y_2)) \, , \ee
for all $x_1, x_2, y_1, y_2 \in \mathfrak v$. Now restricting $p$ to $\p$, we get
\be p((x_1,-x_1),(y_1,-y_1)) = f \left ( p_1(x_1, y_1), p_2(-(x_1, y_1)) \right ) = \tilde{p}(x_1,y_1) \in \K_{k,\ell} [\mathfrak v, \mathfrak v]^{\mathfrak v}\, , \ee
for all $x_1, y_1 \in \mathfrak v$. Now, it is obvious that every polynomial in $\K_{k,\ell} [\mathfrak v, \mathfrak v]^{\mathfrak v}$ can be obtained as the restriction to $\p$ of a polynomial in $\K_{k,\ell}[\g,\g]^\g$; {\it e.g.} take $p_2=0$, $f= \id$ and $p_1 = \tilde{p}$.

We are now going to consider the different symmetric simple Lie algebras listed in \ref{restriction}. Let us first consider the symmetric simple Lie algebras of type AI$_{n}$ for all $n >2$. In this case, we have $\g=\mathfrak{su}(n)$ endowed with an involutive automorphism $\theta$ given by complex conjugation, {\it i.e.} $\theta(x)=\bar{x}$, for all $x\in \mathfrak{su}(n)$. The fixed points of $\theta$ are traceless real antisymmetric matrices which generate an $\mathfrak{so}(n)$ subalgebra. We thus have the symmetric decomposition $\mathfrak{su}(n) = \mathfrak{so}(n) \oplus \p$, where the orthogonal complement $\p$ is the left $\mathfrak{so}(n)$-module generated by the traceless imaginary symmetric matrices of $\mathfrak{su}(n)$.
It follows from the first fundamental theorem for $\mathfrak{so}(n)$-invariant polynomials on $n\times n$ matrices, \cite{ProcesiPaper}, that $\K_{k,\ell}[\p, \p]^{\mathfrak{so}(n)}$ is generated by the following polynomials
\be \label{sopoly} (x,y) \in \p \times \p \rightarrow \tr \, P(x, y) \, , \ee
for all $(i,j)$-homogeneous noncommutative polynomial $P \in \K_{i,j}[X,Y]$, with $i \leq k$ and $j \leq \ell$. The polynomials defined in (\ref{sopoly}) are obviously restrictions to $\p$ of $\mathfrak{su}(n)$-invariant polynomials on ${\mathfrak{su}(n)}$ as, for all $P \in \K_{i,j}[X,Y]$ and all $x, y \in {\mathfrak{su}(n)}$,
\be (x,y) \rightarrow \tr \, P(x,y) \ee
defines an element in $\K_{ m, n} [{\mathfrak{su}(n)}, {\mathfrak{su}(n)}]^{\mathfrak{su}(n)}$. This proves lemma \ref{restriction} for simple symmetric Lie algebras of type $AI_{n>2}$. It is worth noting that in the case of AI$_{2}$, there exist obstructions to the above result which are related to the existence of a further $\mathfrak{so}(2)$-invariant with appropriate symmetries, namely the pfaffian $(x,y) \in \p \times \p \rightarrow \mbox{Pf}([x,y])$. As the latter is not the restriction to $\p$ of any $\mathfrak{su}(2)$ invariant on $\mathfrak{su}(2)$, lemma \ref{restriction} does not hold in this case.

We now turn to type AII$_n$. In this case, we have $\g = \mathfrak{su}(2n)$ endowed with an involutive automorphism $\theta$ given by the symplectic transpose, {\it i.e.}, for all $x \in \mathfrak{su}(2n)$, $\theta(x) = J x^{t} J$, where $J$ is a non-singular skew-symmetric $2n \times 2n$ matrix such that $J^2=-1$. The fixed point set of $\theta$ constitutes an $\mathfrak{sp}(2n)$ subalgebra and we have the following symmetric decomposition $\mathfrak{su}(2n)=  \mathfrak{sp}(2n) \oplus \p$, where $\p \subset \mathfrak{su}(2n)$ is the left $\mathfrak{sp}(2n)$-module of matrices $x \in \mathfrak{su}(2n)$ such that $\theta(x)=-x$. It follows from the first fundamental theorem for $\mathfrak{sp}(2n)$-invariant polynomials on $2n \times 2n$ matrices, \cite{ProcesiPaper}, that $\K_{k,\ell}[\p, \p]^{\mathfrak{sp}(2n)}$ is generated by the following polynomials
\be (x,y) \in \p \times \p \rightarrow \tr \, P(x,y) \, , \ee
for all noncommutative $(i,j)$-homogeneous polynomial $P \in \K_{i,j}[X,Y]$, with $i\leq k$ and $j \leq \ell$. These polynomials are restrictions to $\p$ of $\mathfrak{su}(2n)$-invariant polynomials on ${\mathfrak{su}(2n)}$ as, for all $P \in \K_{i,j}[X,Y]$ and all $x, y \in {\mathfrak{su}(2n)}$,
\be (x,y) \rightarrow \tr \, P(x,y) \ee
defines an element in $\K_{i, j} [{\mathfrak{su}(2n)}, {\mathfrak{su}(2n)}]^{\mathfrak{su}(2n)}$. This proves lemma \ref{restriction} for simple symmetric Lie algebras of type AII$_{n}$.

We finally consider the symmetric simple Lie algebras of type BDI$_{n,1}$ for all $n >2$. In this case, we have the symmetric pairs $({\mathfrak{so}}(n+1), {\mathfrak{so}}(n))_{n>2}$. We introduce the usual basis of ${\mathfrak{gl}}(n+1)$, {\it i.e.} the $(E_{ij})_{0 \leq i,j \leq n}$ defined as the $(n+1) \times (n+1)$ matrices with a $1$ at the intersection of the $i$-th row and $j$-th column and $0$ everywhere else. The matrices $M_{ij}=E_{ij}-E_{ji}$, $0 \leq i,j \leq n$, constitute a basis of ${\mathfrak{so}}(n+1)$, and of these, the $M_{ij}$ with $1 \leq i,j \leq n$ generate an ${\mathfrak{so}}(n)$ subalgebra. We thus have the symmetric decomposition ${\mathfrak{so}}(n+1) = {\mathfrak{so}}(n)\oplus \p$ where $\p$ is the $n$-dimensional ${\mathfrak{so}}(n)$-module spanned by the $P_i = M_{0,i}$, for all $1\leq i\leq n$. The $P_i$ transform under the fundamental representation ${\bf n}$ of $\mathfrak{so}(n)$, as can be checked from 
\be  M_{ij} \triangleright P_k = [M_{ij}, P_k ] = \delta_{jk} P_i - \delta_{ik} P_j \, ,\ee
for all $1\leq i,j,k \leq n$. This means that we are looking for $SO(n)$-invariant $(k,\ell)$-homogeneous polynomials on $\p \times \p = {\bf n} \times {\bf n}$. For all $n>2$, it follows from the first fundamental theorem for $\mathfrak{so}(n)$-invariant polynomials on vectors, \cite{Procesi, spivak}, that such polynomials only depend on the $SO(n)$ scalars built out of the scalar products of their arguments. Let $q$ be the quadratic form defined on $\p \times \p$ by $q(P_i, P_j) = \delta_{ij}$ for all $1 \leq i,j \leq n$. For all $p \in \K_{k,\ell}[\p,\p]^\h$, there exists a polynomial $f : \K^3 \rightarrow \K$ such that, for all $X, Y \in \p$,
\be
p(X,Y) = f(q(X,X),q(X,Y),q(Y,Y)) \, .
\ee
Now, it is obvious that $q$ is the restriction to $\p$ of the map
\be
\mathfrak{so}(n+1)\times \mathfrak{so}(n+1) \rightarrow  \K \, ; \quad (X,Y)  \rightarrow  -\frac{1}{2} \tr (XY) \, , \nonumber
\ee
which is $\mathfrak{so}(n+1)$-invariant. This proves the result for symmetric simple Lie algebras of type BDI$_{n>2,1}$. It is worth noting that in the case of BDI$_{2,1}$, there exist obstructions to the above result which are related to the existence of a further $SO(2)$ invariant, namely $(X,Y) \in \p \times \p \rightarrow \det(X,Y)$. As the latter is not the restriction to $\p$ of any $\mathfrak{so}(3)$ invariant, lemma \ref{restriction} does not hold in this case.

By virtue of the special isomorphisms between lower rank simple Lie algebras, the list of summands in lemma \ref{restriction} actually includes CII$_{1,1} = $ BDI$_{4,1}$ and BDI$_{3,3}=$ AI$_4$. The latter respectively correspond to the symmetric decompositions $\mathfrak{sp}(4) = \left (\mathfrak{sp}(2) \oplus \mathfrak{sp}(2) \right ) \oplus \p$ and $\mathfrak{so}(6) = \left (\mathfrak{so}(3) \oplus \mathfrak{so}(3) \right ) \oplus \p$.

\end{document}